\documentclass[a4paper,12pt]{article}
\usepackage{amsmath, amssymb}
\usepackage{amsthm,comment,mathtools}
\usepackage{color}
\newtheorem{thm}{Theorem}[section]
\newtheorem{DEF}[thm]{Definition}
\newtheorem{cor}[thm]{Corollary}
\newtheorem{prop}[thm]{Proposition}
\newtheorem{lem}[thm]{Lemma}
\newcommand{\p}{\mathfrak{p}}
\newcommand{\h}{\mathfrak{h}}
\newcommand{\F}{\mathcal{F}}
\newcommand{\n}{\mathfrak{n}}
\newcommand{\R}{\mathbb{R}}

\renewcommand{\P}{\mathcal{P}}
\newif\ifpersonalnote \personalnotetrue
\newcommand{\personalnote}[1]{%
  \ifpersonalnote \footnote{\textcolor{red}{\bf #1}}
  \fi
}
\newcommand{\personal}[1]{%
  \ifpersonalnote \textcolor{blue}{\bf\footnotesize#1}
  \fi
}

\newcommand{\personalnoteOFF}{%
  \ifpersonalnote\personalnotefalse}
\title{Local rigidity of certain actions of nilpotent-by-cyclic groups on the sphere}
\author{Mao Okada}
\personalnoteOFF
\begin{document}
\maketitle
\begin{abstract}
Let $G = SU(n, 1)$, $n \geq 2$ be the orientation-preserving isometry group of the complex hyperbolic space $\mathbb{H}_\mathbb{C}^n$ with an Iwasawa decomposition $G = KAN$.
We prove local rigidity of a family of certain actions of a subgroup $\Gamma \subset AN$ on the imaginary boundary $\partial\mathbb{H}_\mathbb{C}^n = S^{2n-1}$.
\end{abstract}
\section{Introduction}
Let $\Gamma$ be a finitely generated group, and $G$ a topological group.
Consider the space $\mathrm{Hom}(\Gamma, G)$ of homomorphisms from $\Gamma$ into $G$ equipped with the topology induced from the product topology of $G^\Gamma$.
Two homomorphisms $\rho_1, \rho_2 \in \mathrm{Hom}(\Gamma, G)$ are said to be \textit{conjugate} to each other if there exists an element $g \in G$ such that $g\rho_1(\gamma)g^{-1} = \rho_2(\gamma)$ for all $\gamma \in \Gamma$.
A homomorphism $\rho \in \mathrm{Hom}(\Gamma, G)$ is said to be \textit{locally rigid} if the conjugacy class of $\rho$ is a neighborhood of $\rho$.
In this paper, we will study local rigidity in a broader sense:
A family $\mathcal{A} \subset \mathrm{Hom}(\Gamma, G)$ of homomorphisms is \textit{locally rigid} if the set of homomorphisms conjugate to certain elements of $\mathcal{A}$ is a neighborhood of $\mathcal{A}$.
Our main interest is the case where $G = \mathrm{Diff}(M)$, the group of $C^\infty$-diffeomorphisms of a compact manifold $M$ with $C^r$-topology $(r = 0, 1, \dots, \infty)$.
In this case, we say that the family $\mathcal{A}$ of actions of $\Gamma$ on $M$ is $C^r$-\textit{locally rigid}.

A typical example of local rigidity of a family of actions is a result of Ghys \cite{Ghys}, which shows $C^r$-local rigidity $(3 \leq r \leq \infty)$ of the family of group actions defined as follows.
Let $\Gamma = \pi_1(\Sigma_g)$ be the fundamental group of an oriented closed surface of genus $g \geq 2$.
Consider the family of embeddings of $\Gamma$ into $PSL(2, \mathbb{R})$ as cocompact lattices.
The standard action of $PSL(2, \mathbb{R})$ on $S^1 = \mathbb{RP}^1$ induces a family of actions of $\Gamma$ on $S^1$, which is locally rigid.
Note that the family is locally compact since $\Gamma$ is a finitely generated group and $PSL(2, \mathbb{R})$ is a finite-dimensional Lie group.

Asaoka \cite{Asaoka} proved $C^2$-local rigidity of a family of actions of the solvable group
$BS(n, k) = \langle a, b_1, \dots, b_n \mid ab_i a^{-1} = b_i^k, b_ib_j = b_j b_i\,\;(i,j = 1, \dots, n) \rangle, \,(n \geq 2, k \geq 2)$ on the sphere $S^n$.
The family of actions are induced by a family of embeddings $\iota:\Gamma \to SO(n+1, 1)$ of $\Gamma$ into $SO(n+1,1)$, which acts on the real hyperbolic space $\mathbb{H}^{n+1}_{\mathbb{R}}$ by isometries in a canonical manner and thus on the imaginary boundary $\partial\mathbb{H}^{n+1}_{\mathbb{R}} = S^n$.
We remark that the images $\iota(\Gamma)$ of the embeddings is not discrete in $SO(n+1,1)$.
It should be mentioned that the statement also holds for $n = 1$, while the family is contained in a single conjugacy class, i.e., such actions are locally rigid.
This fact is a part of the work of Burslem-Wilkinson \cite{BW}.

The result of Asaoka with $\mathbb{H}_\mathbb{R}$ replaced by $\mathbb{H}_\mathbb{C}$ is the statement of our theorem; we proved local rigidity of a family of actions of a solvable group $\Gamma$ on the imaginary boundary $\partial\mathbb{H}^{n+1}_{\mathbb{C}} = S^{2n+1}$ of the complex hyperbolic space.
The groups $\Gamma$ are defined as follows.
Let $SU(n+1, 1) = KAN$ be an Iwasawa decomposition, and $\Lambda$ a lattice of $N$.
As $N$ is a normal subgroup of $AN$, $A$ acts on $N$ by conjugation.
Fix a nontrivial element $a \in A$ which preserves $\Lambda \subset N$, i.e.,  $a\Lambda a^{-1} \subset \Lambda$.
Note that since $A$ is $1$-dimensional, the subgroup $\mathbb{Z} = \langle a \rangle \subset A$ generated by $a$ is a lattice of $A$.
Then we obtain a subgroup $\Gamma = \mathbb{Z} \ltimes \Lambda$ of $SU(n+1, 1)$.
Such a $\Gamma$ will be called a \textit{standard subgroup} of $SU(n+1, 1)$.
Note that $\Gamma$ is not discrete in $AN$.
In fact, the closure is isomorphic to $\mathbb{Z} \ltimes N$.

The action of a standard subgroup on $S^{2n+1}$ induced by the natural action of $SU(n+1, 1)$ on $\partial\mathbb{H}^{n+1}_{\mathbb{C}} = S^{2n+1}$ will be called a \textit{standard action}.
For a fixed standard subgroup $\Gamma \subset SU(n+1, 1)$, an embedding of $\Gamma$ into $SU(n+1, 1)$ as a standard subgroup is not unique.
Thus the family of such embeddings induces that of standard actions of $\Gamma$.
Now we state our main theorem.
\begin{thm}\label{MainThm}
Let $\Gamma$ be a standard subgroup of $SU(n+1, 1)$.
The family of standard actions of $\Gamma$ on $S^{2n+1}$ is $C^3$-locally rigid.
\end{thm}

Note that when we apply the above construction for $SU(n+1,1) =KAN$ to an Iwasawa decomposition $SO(n+1, 1) = KAN$, we obtain the family of group actions in the result of Asaoka.

Our strategy for the proof, which is similar to that of Asaoka's, can be described as follows.
Since the standard actions of $\Gamma$ admit a global fixed point in $S^{2n+1}$, the family of standard actions induce a family of homomorphisms from $\Gamma$ into the group $\mathcal{G}(\R^{2n+1}, \mathrm{O})$ of germs of diffeomorphisms of $\R^{2n+1}$ defined around $\mathrm{O} \in \R^{2n+1}$ and fixing $\mathrm{O}$.
The first step is to reduce our main theorem to local rigidity of the family in $\mathrm{Hom}(\Gamma, \mathcal{G}(\R^{2n+1}, \mathrm{O}))$, which is, so to speak,  local rigidity of local actions.
The Taylor expansion at $\mathrm{O}\in\R^{2n+1}$ induces a homomorphism from $\mathcal{G}(\R^{2n+1}, \mathrm{O})$ onto the group of formal transformations $\mathcal{F}(\R^{2n+1}, \mathrm{O})$ of $\R^{2n+1}$ fixing $O$.
The second step is to show that local rigidity of the family in $\mathrm{Hom}(\Gamma, \mathcal{G}(\R^{2n+1}, \mathrm{O}))$ follows from that of the induced family in $\mathrm{Hom}(\Gamma, \mathcal{F}(\R^{2n+1}, \mathrm{O}))$.
The last step is to prove local rigidity of the family in $\mathrm{Hom}(\Gamma, \mathcal{F}(\R^{2n+1}, \mathrm{O}))$.

A difficulty of the case $G = SU(n+1, 1)$, compared to the case $G = SO(n+1, 1)$, comes from differences in the dynamics of $A$ around the fixed point, where we fixed Iwasawa decomposition $G = KAN$.
For example, in \cite{Asaoka}, Asaoka applied a tool from the theory of dynamical systems called linearization to diffeomorphisms close to the action of a non-trivial element $a \in A$ , while in our case such diffeomorphisms cannot be linearized immediately.
So we used Sternberg's normalization \cite{Sternberg}; a modified version of linearization.

In Section \ref{std}, we will study the standard subgroups of $SU(n+1, 1)$.
In particular, we will give an explicit presentation of such a group.
In Section \ref{Stb}, we will begin with reviewing a proof of linearization of a diffeomorphism around a contracting fixed point.
The goal of Section \ref{Stb} is to give a proof of a normalization which will be used later.
In Section \ref{Heis}, we will set up terminology for spaces on which the Lie group $N$ acts simply transitively, called Heisenberg spaces.
The remaining sections are devoted to the proof of Theorem \ref{MainThm}.

\section{The family of standard actions}\label{std}
\subsection{The action of $SU(n+1, 1)$ on $S^{2n+1}$}\label{stdsu}
Let $SU(n + 1, 1) \subset GL(n+2, \mathbb{C})$ be the group of special linear transformations of $\mathbb{C}^{n+2} = \{ (z_0, \dots, z_{n+1}) \}$ preserving the Hermitian form
\[
\langle z, w \rangle =z_0\overline{w}_{n+1} + \sum_{i =1}^n z_i\overline{w}_i + z_{n+1}\overline{w}_0.
\]
Fix the Iwasawa decomposition $SU(n+1, 1) = KAN$, where\personalnote{$\mathfrak{u}(n+1, 1)\cap \mathfrak{u}(n+2)$ is \[
\left\{ \left.\left(
\begin{array}{ccc}
	ix-iy &-\overline{z}^T & iy\\
	z& U &z\\
	iy&-\overline{z}^T&ix-iy
\end{array} \right) \right| x,y \in \R, z \in \mathbb{C}^n, U \in \mathfrak{u}(n)\right\},
\] 
which is isomorphic to $\mathfrak{u}(n+1) \oplus \mathfrak{u}(1)$ via
\[
\left(
\begin{array}{ccc}
	ix-iy &-\overline{z}^T & iy\\
	z& U &z\\
	iy&-\overline{z}^T&ix-iy
\end{array} \right) \longmapsto \left(\left(
\begin{array}{cc}
	ix &-\overline{z}^T\\
	z& U 
\end{array} \right) , iy \right)
\]
}
\begin{align*}
K &=  \left\{ g \in SU(n+1, 1) \mid \overline{g}^T = g^{-1} \right\}, \\
A &=  \left\{ \left.\left(
\begin{array}{ccc}
	e^s & 0& 0\\
	0& I_n &0\\
	0&0&e^{-s}
\end{array} \right) \right| s \in \R \right\}, \\
N &=  \left\{\left. \left(
\begin{array}{ccc}
	1 &-\overline{z}^T & -\|z\|^2/2 - it\\
	0& I_n &z\\
	0&0&1
\end{array} \right) \right| z \in \mathbb{C}^n, t \in \R \right\}.
\end{align*}
The group $SU(n+1,1)$ naturally acts on the imaginary boundary $\partial\mathbb{H}_\mathbb{C}^{n+1}$ of the complex hyperbolic space.
In fact, since $SU(n+1,1)$ preserves the Hermitian form on $\mathbb{C}^{n+2}$, it also acts on the light cone $L =\{z\in\mathbb{C}^{n+2} | \langle z, z \rangle =0\}$ and its projectivization $\partial\mathbb{H}_\mathbb{C}^{n+1} = \mathbb{P}(L)\subset\mathbb{PC}^{n+1}$, which is diffeomorphic to the $(2n+1)$-dimensional sphere $S^{2n+1}$.
This natural action of $SU(n+1, 1)$ on $S^{2n+1}$ will be denoted by $\rho^0$.

The action $\rho^0$ of $G = SU(n+1, 1)$ on $S^{2n+1}$ can be described by the induced homomorphism $\rho^0_*:\mathfrak{g} = \mathrm{Lie}(G) \to \mathfrak{X}(S^{2n+1})$ of Lie algebras, where $\mathfrak{X}(S^{2n+1})$ is the Lie algebra of the smooth vector fields on $S^{2n+1}$.
To see the structure of the Lie algebra $\mathfrak{g} \subset \mathfrak{gl}(n+2, \mathbb{C})$ of $G$, put
\[
E = \left(
\begin{array}{ccc}
	1 & 0& 0\\
	0& 0 &0\\
	0&0&-1
\end{array} \right) \in \mathfrak{g}.
\]
Then $\mathfrak{g}$ is decomposed as $\mathfrak{g} = \bigoplus_{|\lambda| \leq 2, \lambda \in \mathbb{Z}} \mathfrak{g}^{(\lambda)}$, where $\mathfrak{g}^{(\lambda)} = \{ X \in \mathfrak{g} \mid [E, X] = \lambda X\}$.
More explicitly,
\[
\mathfrak{g}^{(0)} = \left\{ \left.\left(
\begin{array}{ccc}
	z & 0& 0\\
	0& U &0\\
	0&0& -\overline{z}
\end{array} \right) \right| z \in \mathbb{C} , U \in \mathfrak{u}(n), z + \mathrm{tr} \,U -\overline{z} = 0\right\},
\]
\[
\mathfrak{g}^{(-1)} = \left\{ \left.\left(
\begin{array}{ccc}
	0 & 0& 0\\
	\xi& 0 &0\\
	0&-\overline{\xi}^{T}&0
\end{array} \right) \right| \xi \in \mathbb{C}^n \right\}, 
\mathfrak{g}^{(+1)} = \left\{ \left.\left(
\begin{array}{ccc}
	0 & -\overline{\xi}^{T}& 0\\
	0& 0 &\xi\\
	0&0&0
\end{array} \right) \right| \xi \in \mathbb{C}^n \right\},
\]
\[
\mathfrak{g}^{(-2)} = \left\{ \left.\left(
\begin{array}{ccc}
	0 & 0& 0\\
	0& 0 &0\\
	-i\tau&0&0
\end{array} \right) \right| \tau \in \R \right\}, 
\mathfrak{g}^{(+2)} = \left\{ \left.\left(
\begin{array}{ccc}
	0 & 0& -i\tau\\
	0& 0 &0\\
	0&0&0
\end{array} \right) \right| \tau \in \R \right\}.
\]
Note that $\mathfrak{a} = \R E$, $\mathfrak{n} = \mathfrak{g}^{(+1)} \oplus \mathfrak{g}^{(+2)}$, where $\mathfrak{a}$ and $\mathfrak{n} $ are the Lie algebras of $A$ and $N$, respectively.
The vector field $\rho^0_*(E)$ vanishes at the points $p^0 = [1, 0, \dots, 0] \in \mathbb{P}(L) = S^{2n+1}$ and $q^0 = [0, \dots, 0, 1]$.
We will use the atlas $\{(\phi^0, S^{2n+1}\setminus\{q^0\}), (\psi^0, S^{2n+1}\setminus\{p^0\})\}$ defined by
\begin{align*}
\phi^0&:[1, z_1, \dots, z_{n+1}] \longmapsto (\mathrm{Re} z_1, \mathrm{Im} z_1, \dots, \mathrm{Re} z_n, \mathrm{Im} z_n, \mathrm{Im}z_{n+1}), \\
\psi^0&:[z_0, z_1, \dots, z_n, 1] \longmapsto (\mathrm{Im}z_0, \mathrm{Re} z_1, \mathrm{Im} z_1, \dots, \mathrm{Re} z_n, \mathrm{Im} z_n).
\end{align*}
These coordinate charts induce the homomorphisms $\phi^0_*, \psi^0_*: \mathfrak{X}(S^{2n+1}) \to \mathfrak{X}(\R^{2n+1})$ of Lie algebras.
The following lemma is a consequence of a straight-forward computation.
To reduce the notation, put $\partial_i = \partial/{\partial x_i} \in \mathfrak{X}(\R^{2n+1})$.
\begin{lem}\label{stdalg}
\begin{enumerate}
\item $\phi^0_*\circ \rho^0_*(E) = -x_1\partial_1 - \dots -x_{2n}\partial_{2n} -2 x_{2n+1}\partial_{2n+1}$
\item $\phi^0_*\circ \rho^0_*(\mathfrak{g}^{(-1)})$ is generated by $\partial_{2i-1} - x_{2i}\partial_{2n+1}, \partial_{2i} + x_{2i -1}\partial_{2n+1}$ $(1 \leq i \leq n)$.
\item $\phi^0_*\circ \rho^0_*(\mathfrak{g}^{(-2)})$ is generated by $\partial_{2n+1}$.
\end{enumerate}
\end{lem}
\personal{
\begin{proof}
Put $[1, z_1, \dots, z_{n+1}] := (\phi^0)^{-1}(x_1, \dots, x_{2n+1})$, $z: = (z_1, \dots, z_n) \in \mathbb{C}^n$.
(i) Use
\[
\exp tE =
\left(
\begin{array}{ccc}
	e^t & &\\
	& I_n &\\
	& & e^{-t}
\end{array} \right) \in SU(n+1, 1).
\]
\[
\exp tE
\left[
\begin{array}{c}
	1\\
	z\\
	z_{n+1}
\end{array} \right] = 
\left[
\begin{array}{c}
	e^t\\
	z\\
	e^{-t}z_{n+1}
\end{array} \right] = 
\left[
\begin{array}{c}
	1\\
	e^{-t}z\\
	e^{-2t}z_{n+1}
\end{array} \right] \xmapsto {\phi^0}  (e^{-t}\mathrm{Re} z_1, \dots, e^{-t}\mathrm{Im} z_n, e^{-2t}\mathrm{Im}z_{n+1})
\]
\[
 \xmapsto {\frac{d}{dt}|_{t = 0}}(-x_1, \dots, -x_{2n}, -2x_{2n+1}).
\]
(ii) Use
\[
\exp t \left(
\begin{array}{ccc}
	0 & 0& 0\\
	\xi& 0 &0\\
	0&-\overline{\xi}^{T}&0
\end{array} \right) =
\left(
\begin{array}{ccc}
	1 &0 & 0\\
	t\xi& I_n &0\\
	O(t^2)&-t\overline{\xi}^{T}&1
\end{array} \right).
\]
\[
\left(
\begin{array}{ccc}
	1 &0 & 0\\
	t\xi& I_n &0\\
	O(t^2)&-t\overline{\xi}^{T}&1
\end{array} \right)
\left[
\begin{array}{c}
	1\\
	z\\
	z_{n+1}
\end{array} \right] =
\left[
\begin{array}{c}
	1\\
	t\xi + z\\
	O(t^2) -t\overline{\xi}^{T} z + z_{n+1}
\end{array} \right] =  \mathrm{const.} + t
\left[
\begin{array}{c}
	0\\
	\xi\\
	-\overline{\xi}^{T} z
\end{array} \right] + O(t^2)
\]
\[
\xmapsto{\frac{d}{dt}\phi^0 \text{ at } t = 0}
\begin{cases}
\partial_{2j-1} +\mathrm{Im}(-e_j^{T} z)\partial_{2n+1} & ( \text{if } \xi = e_j ) \\
\partial_{2j} +\mathrm{Im}(ie_j^{T} z)\partial_{2n+1} & ( \text{if } \xi = ie_j )
\end{cases}
= \begin{cases}
\partial_{2j-1} -x_{2j}\partial_{2n+1}& ( \xi = e_j ) \\
\partial_{2j} +x_{2j-1}\partial_{2n+1}& (  \xi = ie_j )
\end{cases}
\]
(iii) Use
\[
\exp t \left(
\begin{array}{ccc}
	0 & 0& 0\\
	0& 0 &0\\
	-i&0&0
\end{array} \right) =
\left(
\begin{array}{ccc}
	1 &0 & 0\\
	0& I_n &0\\
	-it&0&1
\end{array} \right).
\]
\[
\left(
\begin{array}{ccc}
	1 &0 & 0\\
	0& I_n &0\\
	-it&0&1
\end{array} \right) \left[
\begin{array}{c}
	1\\
	z\\
	z_{n+1}
\end{array} \right] = \mathrm{const.} + t
\left[
\begin{array}{c}
	0\\
	0\\
	-i
\end{array} \right] \mapsto \mathrm{Im}(-i)\partial_{2n+1} = -\partial_{2n+1}.
\]
\end{proof}}

\subsection{The standard subgroups of $SU(n+1, 1)$}
Let $SU(n+1, 1) = KAN$ be an Iwasawa decomposition.
In this paper, we will study certain finitely-presented subgroups $\Gamma$ of $AN$.
\personalnote{$AN$ does not admit a lattice.
Thus we consider subgroups obtained as, so to speak,  products of lattices in $A$ and $N$.}
\begin{DEF}\label{stdsubgrp}
A subgroup $\Gamma \subset SU(n+1, 1)$ is said to be a \textup{standard subgroup} of $SU(n+1, 1)$ if there exist
\begin{itemize}
\item an Iwasawa decomposition  $SU(n+1, 1) = KAN$,
\item a lattice $\Lambda$ of $N$, and
\item a nontirivial element $a \in A$ with $a\Lambda a^{-1} \subset \Lambda$,
\end{itemize}
such that $\Gamma$ is generated by $a$ and $\Lambda$.
\end{DEF}

In the rest of this subsection, we will give explicit presentations of standard subgroups of $SU(n+1, 1)$.
It suffices to consider the Iwasawa decomposition as in \ref{stdsu}.
We will first describe presentations of lattices in $N$.
Let $\Lambda$ be a lattice of $N$.
Using a result of Mal'cev \cite{Malcev} for general connected simply-connected nilpotent Lie group, we obtain a basis $\{X_1, \dots, X_{2n+1}\}$ of the Lie algebra $\n$ of $N$ such that
\begin{enumerate}
\item $\{\exp X_1, \dots, \exp X_{2n+1}\}$ is a system of generators of $\Lambda$,
\item $\langle X_{i+1}, \dots, X_{2n+1}\rangle $ is an ideal of $\langle X_i, \dots, X_{2n+1}\rangle$ for $i = 1, \dots ,2n$, and
\item $[X_i, X_j] = \sum_k m_{ij}^k X_k$  with some rational constants $m_{ij}^k$.
\end{enumerate}  
We will identify the Lie algebra $\n = \mathfrak{g}^{(+1)}\oplus\mathfrak{g}^{(+2)}$ with $\mathbb{C}^n \oplus \R$ equipped with the bracket
\[
[(\xi, \tau), (\xi',\tau)] = (0, -2\Phi(\xi, \xi')),
\]
where $\Phi(\xi, \xi') = \mathrm{Im}(\overline{\xi'}^T\xi)$.
Then the Lie group $N$ can be identified with the Lie group $\mathbb{C}^n \times \mathbb{R} = \{(z,t)\}$ equipped with the product $(z, t) \cdot (z', t') = (z +z', t + t' - \Phi(z, z'))$.
By the condition(ii), we see that $X_{2n+1} \in [\mathfrak{n}, \mathfrak{n}] = \mathfrak{g}^{(+2)}$.
So $X_{2n+1} = (0, \tau)$ for some non-zero $\tau \in\R$.
Thus for $i =1, \dots, 2n$, $X_i = (\xi_i, \tau_i)$, where $\{\xi_1, \dots, \xi_{2n}\}$ is an $\R$-linear basis of $\mathbb{C}^n$.
Since $\mathrm{Ad}((z,t))(\xi, \tau) = (\xi, \tau - 2\Phi(z, \xi))$, there exists $g = (z, 0) \in N$ such that $\mathrm{Ad}(g)X_i = (\xi_i, 0)$ for $i = 1, \dots, 2n$.
Replacing $X_i$ with $\mathrm{Ad}(g)X_i $ if necessary, we may assume $\tau_i = 0$.
It follows that $m_{ij}^k = 0$ for $k \neq 2n+1$.
Set $m_{ij} = m_{ij}^{2n+1}$.
We may assume, by multiplying $X_{2n+1}$ by the inverse an integer if necessary, that $m_{ij}$ are integers.
Put $b_i = \exp X_i$ for $i = 1, \dots, 2n$ and $c = \exp X_{2n+1}$.
Then $\Lambda$ can be presented as follows:
\[
\Lambda = \langle b_1, \dots, b_{2n}, c \mid [b_i, b_j] = c^{m_{ij}}\text{ for }i,j = 1, \dots, 2n \rangle.
\]
To describe the matrix $(m_{ij})_{ij}$, fix the $\R$-linear basis $\{e_1, ie_1, \dots, e_n, ie_n\}$ of $\mathbb{C}^n$.
Then by the condition (iii), $2G^TJG = \tau(m_{ij})_{ij}$, where $G \in GL(2n, \R)$ correspends to the $\R$-linear basis $\{\xi_1, \dots, \xi_{2n}\}$ and $J \in GL(2n,\mathbb{R})$ corresponds to the multiplication by $i \in \mathbb{C}$ on $\mathbb{C}^n$.
Thus $(m_{ij})_{ij}$ is contained in the set
\[
\mathcal{M} = \{G^TJG \mid G \in GL(2n,\R)\} \cap M(2n, \mathbb{Z}),
\]
where $M(2n, \mathbb{Z})$ denotes the set of integer matrices.
Conversely, given a matrix $(m_{ij})_{ij} \in \mathcal{M}$, the group $\Lambda$ defined by the above presentation can be embedded into $N$ as a lattice. 

Since $\mathrm{Ad}(\exp sE)(\xi, \tau) = (e^s\xi, e^{2s}\tau)$ for $(\xi, \tau) \in \n$, the condition $e^s \in \mathbb{Z}$ is necessary for $\mathrm{Ad}(\exp sE)(\Lambda) \subset \Lambda$.
With the above presentation of $\Lambda$, we see that a standard subgroup $\Gamma$ of $SU(n+1, 1)$ has the following presentation:
\begin{equation}\label{nbc}
\Gamma = \langle a, b_1, \dots, b_{2n}, c \,|\, ab_ia^{-1}=b_i^k, \, aca^{-1}= c^{k^2}, \, [b_i, b_j] = c^{m_{ij}}\rangle,
\end{equation}
where $(m_{ij})_{ij} \in \mathcal{M}$ and $k \in \mathbb{Z}, k \geq 2$.

\subsection{The standard actions of a standard subgroups}
Let $\Gamma$ be a standard subgroup of $SU(n+1, 1)$.
By the action $\rho^0$ of $SU(n+1, 1)$ on $S^{2n+1}$, each embedding of $\Gamma$ into $SU(n+1, 1)$ as a standard subgroup induces an action of $\Gamma$ on $S^{2n+1}$.
We will call such an action of $\Gamma$ a \textit{standard action} of $\Gamma$ on $S^{2n+1}$.
The following is the main theorem of the present paper.

\setcounter{thm}{0}
\setcounter{section}{1}

\begin{thm}
Let $\Gamma$ be a standard subgroup of $SU(n+1, 1)$.
The family of standard actions of $\Gamma$ on $S^{2n+1}$ is $C^3$-locally rigid.
\end{thm}

\setcounter{thm}{2}
\setcounter{section}{2}

From now on, we fix a standard subgroup $\Gamma$ of $SU(n+1, 1)$ with the presentation (\ref{nbc}).
The action of $\Gamma$ on $S^{2n+1}$ induced by the inclusion $\Gamma \subset SU(n+1, 1)$ will also be denoted by $\rho^0$.
To prove Theorem\ref{MainThm}, it is enough to show that for each embedding $\iota$ of $\Gamma$ into $SU(n+1, 1)$ as a standard subgroup, any actions sufficiently close to $\rho^0\circ\iota$ are conjugate to standard ones.
Replacing $\Gamma$ with its embedded image $\iota(\Gamma)$ if necessary, it suffices to show that actions close to $\rho^0$ are conjugate to standard ones.
Moreover, replacing $\rho^0$ with its conjugacy, we may assume that the Iwasawa decomposition appeared in Definition \ref{stdsubgrp} is the one given in \ref{stdsu}.

In the remainder of this subsection, we will give an explicit description of the action $\rho^0$ of $\Gamma$ in terms of the coordinate charts $\phi^0, \psi^0$ introduced in Subsection \ref{stdsu}.
Note that the map $\rho^0(a)$ has exactly one contracting fixed point $p^0 \in S^{2n+1}$ and exactly one expanding fixed point $q^0 \in S^{2n+1}$, while $\rho^0(\lambda)$ has exactly one fixed point $p^0$ for all $\lambda \neq 1 \in \Lambda$.
In particular, $p^0$ is the \textit{global fixed point} of $\rho^0$, i.e., $\rho^0(\gamma)p^0 = p^0$ for all $\gamma \in \Gamma$.
Thus the coordinate chart $\psi^0$ on $S^{2n+1}\setminus\{p^0\}$ induces an action of $\Gamma$ on $\R^{2n+1}$, which will be denoted by $\psi^0_*\rho^0$.
Then it is easy to see the following, the proof of which is left to the reader.
\begin{lem}\label{std1}
There are bases $\{u_1, \dots, u_{2n}\}$, $\{v_1, \dots, v_{2n}\}$ of $\R^{2n}$ and $t \in \R$ such that for any $(x_0, x) = \R \times \R^{2n} = \R^{2n+1}$,
\begin{align*}
\psi^0_*\rho^0(a)(x_0, x) &= (k^2x_0, kx), \\
\psi^0_*\rho^0(b_i)(x_0, x) &= (x_0 - \langle u_i, x \rangle, x +v_i), \\
\psi^0_*\rho^0(c)(x_0, x) &= (x_0 - t, x),
\end{align*}
where $\langle\,\cdot, \cdot\,\rangle$ denotes the Euclidean inner product on $\R^{2n}$.
\end{lem}
\personal{
\begin{proof}
Recall Lemma \ref{stdalg} and its proof.
Then we see that $\psi^0_*\circ\rho^0_*(E)$ is similar to $\phi^0_*\circ\rho^0_*(E)$ and $\psi^0_*\circ\rho^0_*(\mathfrak{g}^{(+1)})$ is similar to $\phi^0_*\circ\rho^0_*(\mathfrak{g}^{(-1)})$.
Then,  $\psi^0_*\circ\rho^0_*(\mathfrak{g}^{(+2)})$ is similar to $\phi^0_*\circ\rho^0_*(\mathfrak{g}^{(-2)})$. 
Note that the flow generated by $\partial_{2j-1} -x_{2j}\partial_{2n+1}$, or more simply, $\partial_1 -x_2\partial_3$ is $(x_1, x_2, x_3) \mapsto (x_1 +t, x_2, x_3 -tx_2)$.
\end{proof}
}
In general, for a diffeomorphism $f$ that is defined around $\mathrm{O} \in \R^m$, and fixes $\mathrm{O}$, the differential $(df)_\mathrm{O}$ at $\mathrm{O} \in \R^m$ can be identified with a matrix in $GL(m, \R)$.
For each $\rho^0(\gamma)\;(\gamma \in \Gamma)$, the coordinate chart $\phi^0$ induces a diffeomorphism around $\mathrm{O} \in \R^{2n+1}$ fixing $\mathrm{O}$, which will be denoted by $\phi^0_*(\rho^0(\gamma))$.
A straightforward computation shows the following, the proof of which is left to the reader.

\begin{lem}\label{std2}
\begin{enumerate}
\item The diffeomorphism $\phi^0_*(\rho^0(a))$ is linear:
\[
\phi^0_*(\rho^0(a))(x) = \left(
\begin{array}{cc}
\frac{1}{k}I_{2n}& 0\\
0&\frac{1}{k^2}
\end{array} \right)x\]
for $x$ around $\mathrm{O} \in \R^{2n+1}$.
\item There is a basis $\{u_1, \dots, u_{2n}\}$ of $\R^{2n}$ such that 
\[ (df_i)_\mathrm{O} = \left(
\begin{array}{cc}
I_{2n}& u_i\\
0&1
\end{array} \right), \;
(dg)_\mathrm{O} = I_{2n+1},
\]
where $f_i = \phi^0_*(\rho^0(b_i))$, $g = \phi^0_*(\rho^0(c))$.
\end{enumerate}
\end{lem}
\personal{
\begin{proof}
(i) follows from Lemma \ref{stdalg}.
To show (ii), the equation for $b$ is sufficient; the commutator of such matrices are trivial, and $[b_i, b_j] = c^{m_{ij}}$.
\[
\left( \begin{array}{ccc}
 1 & -\overline{\xi}^T &-\frac{1}{2}\|\xi\|^2 \\
 0 &I_n &\xi \\
 0 &0 &1
\end{array} \right)
\left[ \begin{array}{c}
 1\\ z \\ z_{n+1}
\end{array} \right ] = 
\left[ \begin{array}{c}
 1 -t\overline{\xi}^Tz -\frac{1}{2}\|\xi\|^2z_{n+1} \\ z + z_{n+1}\xi\\ z_{n+1}
\end{array} \right ]
\]
By the formula
\[
\frac{1}{1-r} = 1 + r +r ^2 + \dots,
\]
we obtain
\[
\frac{1}{1 -\overline{\xi}^Tz +\frac{1}{2}\|\xi\|^2z_{n+1}} = 1 +\overline{\xi}^Tz -\frac{1}{2}\|\xi\|^2z_{n+1} + O(x^2).
\]
Thus
\[
= \left[ \begin{array}{c}
 1 \\ (1 +\overline{\xi}^Tz +\frac{1}{2}\|\xi\|^2z_{n+1})(z + z_{2n+1}\xi) + O(x^2)\\ (1 +\overline{\xi}^Tz +\frac{1}{2}\|\xi\|^2z_{n+1})z_{2n+1} +O(x^2)
\end{array} \right ] = 
\left[ \begin{array}{c}
 1 \\ z +ix_{2n+1}\xi+ O(x^2)\\ ix_{2n+1} +O(x^2)
\end{array} \right ],
\]
where we used $z_{n+1} = ix_{2n+1} + O(x^2)$.
Applying $\phi^0$, we proved the claim.
\end{proof}
}

\section{Sternberg's normalization}\label{Stb}
We first recall Sternberg's original proof of linearization of a diffeomorphism around a contracting fixed point.
Let $\mathcal{G} = \mathcal{G}(\R^m, \mathrm{O})$ be the group of germs of diffeomorphisms defined around $\mathrm{O} \in \R^m = \{(x_1, \dots, x_m)\mid x_i \in \R\}$ and fixing $\mathrm{O} \in \R^m$.
Let $\R[[x_1, \dots, x_m]]$ be the ring of formal power series in the variables $x_1, \dots, x_m$ over $\R$.
The set $\mathcal{F} = \mathcal{F}(\R^m, \mathrm{O}) \subset (\R[[x_1, \dots, x_m]])^m$ of the Taylor expansions of the elements in $\mathcal{G}$ has a natural group structure.
The group $\mathcal{F}$ will be called the group of \textit{formal transformations} of $\R^m$.
Let  $D_\mathrm{O}f \in \mathcal{F}$ be the formal transformation defined by $f \in \mathcal{G}$.
Note that the group $GL(m, \R)$, which is a subgroup of the group of diffeomorphisms around $\mathrm{O}\in\R^{2n+1}$ that fix $\mathrm{O}$, is naturally a subgroup of $\mathcal{G}$ and $\mathcal{F}$.
When $F \in \F$ is the Taylor expansions $F = D_\mathrm{O}f$ of $f \in \mathcal{G}$, the differential $(df)_\mathrm{O}$ at $\mathrm{O} \in \R^m$, which can be identified with a matrix in $GL(m, \R)$, will be called the \textit{linear part} of $F$.

\begin{thm}[Sternberg's linearization \cite{Sternberg}]\label{Stlin}
Let $f \in \mathcal{G}$ be the germ of a contraction at $\mathrm{O} \in \R^m$, $\lambda_1, \dots, \lambda_m \in \mathbb{C}, \;|\lambda_i| < 1$ the eigenvalues of the differential $(df)_\mathrm{O}$ of $f$ at $\mathrm{O} \in \R^m$, and $L \in \mathcal{G}$ the germ of the linear transformation defined by $(df)_\mathrm{O} \in GL(m, \R)$.
Assume the following condition, called the \textup{resonant condition}:
\[
\lambda_i \neq \lambda_1^{l_1} \dots \lambda_m^{l_n} \text 
{ for all } i=1, \dots, m, \;l_j \in \mathbb{Z}_{\geq 0},\; \sum_j l_j \geq 2.
\]
Then $f$ and $L$ are conjugate in $\mathcal{G}$.
\end{thm}

Sternberg's proof of the above theorem consists of the following two propositions.
\begin{prop}[\cite{Sternberg}]\label{St1}
If $f, f' \in \mathcal{G}$ are the germs of contractions satisfying $D_\mathrm{O}f = D_\mathrm{O}f' \in \mathcal{F}$, then $f$and $f'$ are conjugate by an element $h \in \mathcal{G}$ with $D_\mathrm{O}h = 1 = D_\mathrm{O} \mathrm{id}\in \mathcal{F}$.
\end{prop}
\begin{prop}[\cite{Sternberg}]\label{Stlinf}
Let $F \in \mathcal{F}$ be a formal transformation, and $L \in GL(m, \R)$ the linear part of $F$.
If $L$ is a contraction satisfying the resonant condition, then $F$ and $L$ are conjugate in $\mathcal{F}$.
\end{prop}

As mentioned in \cite{Sternberg}, the latter proposition has a natural modification in the case where the resonant condition does not hold.
In fact, under some weaker conditions, one can find a (non-linear) polynomial transformation conjugate to the given transformation.
Such a technique is called normalization of a transformation.
While Proposition \ref{St2} below, which will be used later, is an immediate consequence of \cite{Sternberg}, we will give here a proof of Proposition \ref{St2}.
In fact, the notation and calculation in the proof will also be used later.

Recall the notation introduced in Section \ref{std}.
We want to normalize formal transformations close to the linear one $\phi^0_*(\rho^0(a)) \in GL(2n+1, \R)$ induced by the standard action of $a \in \Gamma$ around $p^0$.
By Lemma \ref{std2},
\[
\phi^0_*(\rho^0(a)) = \left(
\begin{array}{cc}
\frac{1}{k}I_{2n}& 0\\
0&\frac{1}{k^2}
\end{array} \right), \; k \geq 2.
\]
Let us denote the above matrix by $I(k)$.
Observe that $I(k)$ does not satisfy the resonant condition in Theorem \ref{Stlin}.
Thus formal transformations close to $I(k)$, in general, do not satisfy the resonant condition.

Proposition \ref{Stlinf} was proved by solving an equation in $\F = \mathcal{F}(\R^m, \mathrm{O})$ by induction on the degree.
More explicitly, to prove Proposition \ref{Stlinf}, it suffices to find a solution of the equation
\[
FH = HL
\]
in $H \in \F$.
Note that each $F \in \F$ can be uniquely written as an infinite sum of maps in $ \mathcal{S}^{(r)}(\R^m)$, where $\mathcal{S}^{(r)}(\R^m)$, $r \geq 1$ denotes the space of polynomial functions $F:\R^m \to \R^m$ whose components are homogeneous polynomials of degree $r$.
Thus we can construct a solution by induction on $r$.

In our case, instead of the decomposition of $\F = \F(\R^{2n+1}, \mathrm{O})$ into the spaces $\mathcal{S}^{(r)}(\R^{2n+1})$, it is convenient to use another decomposition of $\F$.
Let $\mathcal{F}^{(r)}$, $r \geq -1$ be the collection of polynomial functions $F:\R^{2n+1} \to \R^{2n+1}$ such that $I(k)\circ F \circ I(k)^{-1} = k^r F$ and that $F(\mathrm{O}) = \mathrm{O}$.
More explicitly, $F \in \F^{(r)}$ if and only if $F$ is a polynomial function such that
\begin{itemize}
\item $F_i(x) \in \mathcal{H}^{(r+1)}$, $(1 \leq i \leq 2n)$
\item $F_{2n+1}(x) \in \mathcal{H}^{(r+2)}$,
\end{itemize}
where $F_i(x)$ denotes $i$-th component of $F(x)$, and $\mathcal{H}^{(r)}$ denotes the $\R$-linear span of $x_1^{r_1}\dots x_{2n+1}^{r_{2n+1}}$'s $(r_1 + \dots + r_{2n} + 2r_{2n+1} = r)$, which is, so to speak, the space of homogeneous polynomials of ``weighted'' degree $r$.
Let $\P_0 \subset \F$ be the subgroup defined by
\[
\P_0 = \F \cap \F^{(0)}.
\]
Thus, for $F \in \F^{(0)}$, $F \in \P_0$ if and only if the linear part of $F$ is an invertible matrix.
Then each $F \in \F$ can be written as an infinite sum
\[
F = \sum_{r \geq -1} F^{(r)}
\]
of maps $F^{(r)} \in \F^{(r)}$.
In particular, there is a natural bijection between the set $\F$ and the set $\P_0 \times \prod_{r\geq -1, r \neq 0} \F^{(r)}$.
Note that the map $\mathcal{F} \to \P_0$, $F \mapsto F^{(0)}$ is a group homomorphism.

\begin{prop}\label{St2}
If $F \in \mathcal{F}$ is $C^1$-close to $I(k) \in GL(2n + 1, \R)$, then $F$ is conjugate to $G = F^{(0)}\in \mathcal{F}$ in $\F$.
\end{prop}
\begin{proof}
Suppose that $F$ is $C^1$-close to $I(k)$.
Replacing $F$ with its conjugate by a linear transformation, we may assume that $F^{(-1)} = 0$.
We claim that there is a unique transformation $H \in \mathcal{F}$ such that $H^{(-1)} = 0$, $H^{(0)} = \mathrm{id} \in \P_0$, and $FH = HG$.
We will solve the equation
\[
(FH)^{(r)} = (HG)^{(r)}\;(r \geq -1),\;H^{(-1)} = 0,\;H^{(0)} = \mathrm{id}
\]
by induction on $r$.
When $r = -1$, the both sides of the equation vanish.
The case $r = 0$ follows from the fact that the map $F \mapsto F^{(0)}$ is a group homomorphism from $\mathcal{F}$ onto $\P_0$.
If $r \geq 1$, the equation is equivalent to
\[
L \circ H^{(r)} = H^{(r)}\circ L + \Phi,
\]
where $L$ is the linear part of $G$ and $\Phi$ is a sum of terms determined by $F^{(p)}$ $(p \leq r)$ and $H^{(q)}$ $(q < r)$.
Now the claim follows from Lemma \ref{St2_1} below.
\end{proof}
\begin{lem}\label{St2_1}
Let $\P_0\cap GL(2n+1, \R)$ $(= GL(2n, \R) \times GL(1, \R))$ be the group of linear parts of the transformations of $\P_0$.
There is a neighborhood $\mathcal{U}$ of $I(k) \in \P_0\cap GL(2n+1, \R)$ such that for $L \in \mathcal{U}$ and $r \geq 1$, the linear map
\[
\mathcal{F}^{(r)} \to \mathcal{F}^{(r)},\; F \mapsto L\circ F \circ L^{-1} -F
\]
is a linear automorphism of $\mathcal{F}^{(r)}$.
\personalnote{As $I(k)F I(k)^{-1} = k^r F$, what is nontrivial is uniformness on $r$.}
\end{lem}
\begin{proof}
Let $\|\cdot\|$ be a norm on $\R^{2n+1}$.
Fix a constant $c > 1$ and a neighborhood $\mathcal{U} \subset \P_0\cap GL(2n+1, \R)$ of $(k)$ such that $c^4 < k$ and $\|LI(k)^{-1}\|, \|I(k)L^{-1}\| < c$ for $L \in \mathcal{U}$, where $\|L\| = \max_{\|x\| =1 } \|Lx\|$.
Define a norm on $\mathcal{F}^{(r)}$ by $\|F\| = \max_{\|x\| =1 } \|F(x)\|$.
Since each component of $F \in \mathcal{F}^{(r)}$ is a polynomial in $x_1, \dots, x_{2n+1}$   of degree at most $r+2$, we see that $\| F \circ L\| \leq \max\{\|F\|\|L\|^{r+2}, \|F\|\}$ for $L \in GL(2n+2, \R)$.

We will show that $F = 0$ if $F \in \F^{(r)}$, $r \geq 2$, and $L^{-1}\circ F \circ L = F$.
Since $I(k)\circ F \circ I(k)^{-1} = k^r F$ for all $F \in \mathcal{F}^{(r)}$,
\[
\begin{split}
k^r\|F\| &= \|I(k)\circ F \circ I(k)^{-1}\| = \|I(k)\circ L^{-1}\circ F \circ L  \circ I(k)^{-1}\| \\
&\leq \|I(k)L^{-1}\| \|F\| \|LI(k)^{-1}\|^{r+2} \leq c^{r+3}\| F\|.
\end{split}\]
As $c^{r+3} < k^r$, we see that $\|F\| = 0$.
\end{proof}

\section{Heisenberg space}\label{Heis}
Let $SU(n+1, 1) = KAN$ be the Iwasawa decomposition as in Section \ref{stdsu}.
The goal of this section is to set up terminology for affine spaces modeled on $\n = \mathrm{Lie}(N)$, which will be used in Subsection \ref{persistcycl} and Subsection \ref{persistnilp}.

\begin{DEF}
A smooth manifold $M$ equipped with a simply transitive smooth action $\rho$ of $N$ is called a \textup{Heisenberg space}.
The action $\rho$ will be called a \textup{Heisenberg structure} of $M$.
\end{DEF}

\begin{DEF}
Let $M$ be a smooth manifold with $\dim M = 2n+1$ and $\mathfrak{X}(M)$ be the Lie algebra of smooth vector fields on $M$.
A subalgebra $\mathfrak{h} \subset \mathfrak{X}(M)$ will be called a \textup{Heisenberg connection} if $\h$ is isomorphic to $\n$ as a Lie algebra and $\n$ is a frame field on $M$, i.e., 
\[
T_xM = \{X(x) \mid X \in \h\}
\]
for all $x \in M$.
\end{DEF}

For the homomorphism $\rho_*:\n \to \mathfrak{X}(M)$ induced by a Heisenberg structure, $\rho_*(\n) \subset \mathfrak{X}(M)$ is a Heisenberg connection.
A Heisenberg connection $\h$ is locally spanned by
\[
\partial_{2i-1} - x_{2i}\partial_{2n+1}, \partial_{2i} + x_{2i -1}\partial_{2n+1}, \partial_{2n+1}\;(1 \leq i \leq n)
\]
where $\partial_i = \partial /\partial x_i$.
The subalgebra of $\mathfrak{X}(\R^{2n+1})$ spanned by these vector fields will be called the \textit{standard Heisenberg connection} on $\R^{2n+1}$.
For any Heisenberg space $(M, \rho)$, the Heisenberg connection $\rho_*(\n)$ on $M$ is smoothly  conjugate to the standard Heisenberg connection on $\R^{2n+1}$.

\begin{lem}\label{Hpres}
Let $(M, \rho)$ be a Heisenberg space.
\begin{enumerate}
\item There is a Heisenberg structure $\sigma$ on $M$ which preserves the Heisenberg structure $\rho$, i.e., $\rho(g_1)\sigma(g_2) = \sigma(g_2)\rho(g_1)$ for all $g_1, g_2 \in N$.
\item Let $\sigma$ be a Heisenberg structure preserving $\rho$, $U$ a neighborhood of a point $x \in M$,  and $f:U \to f(U) \subset M$ a diffeomorphism onto its image.
Assume $f$ preserves $\rho$, i.e., $f(\rho(g)x) = \rho(g)f(x)$ if $\rho(g)x, x \in U$.
Then there is a neighborhood $V \subset U$ of $x$ and $g \in N$ such that $f|_V = \sigma(g)|_V$.

\item For any Heisenberg structures $\sigma, \sigma'$ preserving $\rho$, $\sigma(N) = \sigma'(N)$.
\end{enumerate}
\end{lem}
\begin{proof}
It is sufficient to consider a Heisenberg structure $\rho$ on $\R^{2n+1}$ such that $\rho_*(\n)$ is the standard Heisenberg connection.
Let $\h'$ be the  Heisenberg connection on $\R^{2n+1}$ spanned by
\[
\partial_{2i-1} + x_{2i}\partial_{2n+1}, \partial_{2i} - x_{2i -1}\partial_{2n+1}, \partial_{2n+1}\;(1 \leq i \leq n).
\]
It is easy to see that $[X, Y] = 0$ for all $X\in \rho_*(\n)$ and $Y \in\h'$.
Fix an isomorphism $\varphi:n \to \h'$ of Lie algebras.
Then $\sigma(\exp(X)) = \exp(\varphi X)\;(X \in \n)$ defines the Heisenberg structure $\sigma$ preserving $\rho$.

Let $U$ a neighborhood of a point $x \in \R^{2n+1}$,  and $f:U \to f(U) \subset \R^{2n+1}$ a diffeomorphism onto its image which preserves $\rho$.
Let $g \in N$ be the element with $\sigma(g)x = f(x)$.
Then the diffeomorphism $\sigma(g)^{-1} \circ f:U \to \rho(g)^{-1} \circ f(U)$ also preserves $\rho$ and $\sigma(g)^{-1} \circ f(x) = x$.
It follows that $\sigma(g)^{-1} \circ f$ is the identity on a neighborhood of $x$.

Let $\sigma'$ be another Heisenberg structure which preserves $\rho$.
Fix $g \in N$.
By (ii), we obtain a neighborhood of a point of $\R^{2n+1}$ and an element $g' \in N$ such that $\sigma(g) = \sigma'(g')$ on the neighborhood.
Thus $\sigma(g) = \sigma'(g')$ on $\R^{2n+1}$.
\end{proof}
The following is an immediate corollary of the above lemma.
\begin{cor}\label{Hconnpres}
Let $\h$ be a Heisenberg connection on a smooth manifold $M$.
Then the centralizer $Z(\h)$ of $\h$ in $\mathfrak{X}(M)$ is a Heisenberg connection.
\end{cor}

Note that if $\h$ is the standard Heisenberg connection on $\R^{2n+1}$, then $Z(\h)$ is spanned by
\[
\partial_{2i-1} + x_{2i}\partial_{2n+1}, \partial_{2i} - x_{2i -1}\partial_{2n+1}, \partial_{2n+1}\;(1 \leq i \leq n).
\]

The automorphism $f$ of $\n = \mathfrak{g}^{(+1)} \oplus \mathfrak{g}^{(+2)}$ defined by
\[
f(X) =
\begin{cases}
 (\log \lambda)X & (X \in \mathfrak{g}^{(+1)})\\
 2(\log \lambda)X & (X \in \mathfrak{g}^{(+2)})
 \end{cases}
\]
is called the \textit{dilation} by a constant $\lambda >0$.

\begin{DEF}
Let $\h \subset \mathfrak{X}(M)$ be a Heisenberg connection.
A vector field $E \in \mathfrak{X}(M)$ is called a \textup{dilation} by $\lambda > 0$ of $\h$ if $\mathrm{ad}(E)$ preserves $\h$ and the endomorphism on $\h \cong \n$ induced by $\mathrm{ad}(E)$ is the dilation by $\lambda$.
\end{DEF}

Note that the definition of the dilation does not depend on the choice of an isomorphism between $\h$ and $\n$.
By definition, if $E , E' \in \mathfrak{X}(M)$ are dilations of $\h$ by $\lambda >0$, then  $E - E' \in Z(\h)$.

\begin{lem}\label{Hdil}
Let $\h \subset \mathfrak{X}(M)$ be a Heisenberg connection.
There is a dilation $E$ of $\h$ by $\lambda >0$ which is also a dilation of $Z(\h)$ by $\lambda$.
\end{lem}
\begin{proof}
We may assume that $\h$ is the standard Heisenberg connection on $\R^{2n+1}$.
Then 
\[
E = -(\log \lambda)(x_1\partial_1 + \dots + x_{2n}\partial_{2n} + 2x_{2n+1}\partial_{2n+1})
\]
is a dilation of both $\h$ and $Z(\h)$ by $\lambda$.
\end{proof}

\section{Local rigidity of the standard actions} \label{GtoL}
As mentioned in Section \ref{std}, the action $\rho^0$ of $\Gamma$ on $S^{2n+1}$ admits the global fixed point $p^0 \in S^{2n+1}$.
If $\sigma = \rho^0 \circ \iota$ is  a standard homomorphism with the embedded image $\iota(\Gamma)$ is a standard subgroup whose associated Iwasawa decomposition coincides with that of $\Gamma$, then the point $p^0$ is also the global fixed point of $\sigma$.
The coordinate chart $\phi^0$ around $p^0$ induces a group homomorphism $\phi^0_*:\mathrm{Diff}(S^{2n+1}, p^0) \to \mathcal{G} = \mathcal{G}(\R^{2n+1}, \mathrm{O})$, where $\mathrm{Diff}(S^{2n+1}, p^0)$ is the group of  diffeomorphism of $S^{2n+1}$ that fix $p^0$.
For each standard action $\sigma$ as above, $\phi^0$ induces the homomorphism $\phi^0_*\sigma$ of $\Gamma$ into $\mathcal{G}$ defined by $\phi^0_*\sigma(\gamma) = \phi^0_*(\sigma(\gamma))$, $\gamma \in \Gamma$.
Such a homomorphism will be called a \textit{standard homomorphism} from $\Gamma$ into $\mathcal{G}$.

\begin{prop}\label{Main_1}
The family of standard homomorphisms in $\mathrm{Hom}(\Gamma, \mathcal{G})$ is locally rigid, where $\mathcal{G}$ is equipped with $C^3$-topology.
\end{prop}

In this section, we will prove Theorem \ref{MainThm} using Proposition \ref{Main_1}.
The proof of Proposition \ref{Main_1} will be postponed to the next section.

To derive Theorem \ref{MainThm} from  Proposition \ref{Main_1}, we will first prove the following proposition, which  allows us to extend a local conjugacy between local actions to a global conjugacy.
\begin{prop}\label{LtoG}
Let $\rho \in \mathrm{Hom}(\Gamma, \mathrm{Diff}(S^{2n+1}, p^0))$ be an action with a global fixed point $p^0 \in S^{2n+1}$.
Assume that the induced homomorphism $\phi^0_*\rho\in \mathrm{Hom}(\Gamma, \mathcal{G})$ is conjugate to a standard one $\phi^0_*\sigma \in \mathrm{Hom}(\Gamma, \mathcal{G})$, where $\sigma \in \mathrm{Hom}(\Gamma, \mathrm{Diff}(S^{2n+1}, p^0))$ is a standard action.
Then the homomorphism $\rho \in \mathrm{Hom}(\Gamma, \mathrm{Diff}(S^{2n+1}, p^0))$ is conjugate to $\sigma$ in $\mathrm{Hom}(\Gamma, \mathrm{Diff}(S^{2n+1}, p^0))$.
\end{prop}

\begin{proof}
Fix a diffeomorphism $h'$ around $\mathrm{O}\in\R^{2n+1}$ fixing $\mathrm{O}$ whose germ is a conjugacy between $\phi^0_*\rho$ and $\phi^0_*\sigma$.
Using the coordinate chart $\psi^0$ on $ S^{2n+1}\setminus\{p^0\}$ defined in Section \ref{std}, put $K_R =  (\psi^0)^{-1}([ -R, R]^{2n+1})$ for $R > 0$.
Let $U_R = S^{2n+1} \setminus K_R$ be an  open neighborhood of $p^0 \in S^{2n+1}$.
By the assumption, for sufficiently large $R$,
\[
h^{-1} \circ \rho(\gamma) \circ h(x) = \sigma(\gamma)(x) \; ( \gamma \in \Lambda_{\pm1}\cup\{a^{-1}\}, x \in U_R), 
\]
where $h = (\phi^0)^{-1}\circ h' \circ \phi^0$ is a diffeomorphism around $\mathrm{O}\in S^{2n+1}$ fixing $\mathrm{O}$, and $\Lambda_{\pm1} = \{ b_1, b_1^{-1}, \dots, b_{2n}, b_{2n}^{-1}, c, c^{-1} \}$ is the system of generators of $\Lambda$.
For any $x \in S^{2n+1}\setminus\{p^0\}$, by Lemma \ref{std1}, there is an integer $m \in \mathbb{Z}$ with $\sigma(c)^m(x) \in U_R$.
Thus $S^{2n+1} = \bigcup_m\sigma(c)^{-m}(U_R)$.
Let $\bar{h}_m:\sigma(c)^{-m}(U_R) \to S^{2n+1}$ be the smooth map into its image defined by 
\[
\bar{h}_m(x) =  \rho(c)^{-m} \circ h \circ \sigma(c)^m(x).
\]
We will show that $\bar{h}_m(x) = \bar{h}_{m'}(x)$ for $x \in \sigma(c)^{-m}(U_R) \cap \sigma(c)^{-m'}(U_R)$ to obtain the map $\bar{h}:S^{2n+1} \to S^{2n+1}$.
By Lemma \ref{std1} and the definition of $K_R$, we see that if $x \in \sigma(c)^{-m}(U_R) \cap \sigma(c)^{-m'}(U_R)$, then there is a word $w = \gamma_l\dots\gamma_1$ on $\Lambda_{\pm1}$ such that $c^m = wc^{m'} \in \Lambda$ and that $\sigma(w_ic^{m'})(x) \in U_R$ for $1 \leq i \leq l$, where $w_i = \gamma_i\dots\gamma_1$.
\personalnote{
See Lemma \ref{std1}.
Fix a point $(x_0, x) \in \R^{2n+1}$.
Consider the orbit $\Lambda (x_0, x) \subset \R^{2n+1}$.
Note that the action of $N$ (and hence $\Lambda$) on $\R^{2n+1} = S^{2n+1}\setminus\{p^0\}$ is transitive. 
In view of Lemma \ref{std1} and the relations on $\Lambda$, we see that its projection to $\R^{2n}$ is $\{x + \sum_j n_jv_j \mid n_j \in \mathbb{Z}\}$, and that each fiber is of the form $y + t\mathbb{Z}$ for some $y \in \R$.
(Note that $\langle u_i, v_i \rangle$ must be vanished.)
Thus edging the points in $\Lambda (x_0, x)$ properly by segments, we obtain an embedded image $X$ of the Cayley graph of $\Lambda$ with respect to the system of generator $\Lambda_{\pm1}$.
Let $Y$ be a maximal subgraph of $X \setminus [-R, R]^{2n+1}$.
By construction, $Y$ is connected.
}
Then
\[\begin{split}
&\rho(w_ic^{m'})^{-1} \circ h \circ \sigma(w_ic^{m'})(x) \\
&= \rho(w_{i-1}c^{m'})^{-1} \circ \rho(\gamma_i)^{-1} \circ h \circ \sigma(\gamma_i)\circ \sigma(w_{i-1}c^{m'})(x) \\
&= \rho(w_{i-1}c^{m'})^{-1} \circ h \circ \sigma(w_{i-1}c^{m'})(x)
\end{split}\]
for $i = 1, \dots, l$, where $w_0 =1 \in \Lambda$.
It follows that  $\rho(c^{-m}) \circ h \circ \sigma(c^m)(x) = \rho(c^{m'})^{-1} \circ h \circ \sigma(c^{m'})(x)$.
Thus $\bar{h}$ is well-defined.

Since $\bar{h}$ is locally diffeomorphic, $\bar{h}$ is a smooth covering map on $S^{2n+1}$, which is a diffeomorphism.
As $c \in \Lambda$ commutes with any $\gamma \in \Lambda_{\pm1}$,  $\bar{h}_m \circ \sigma(\gamma) (x) = \rho(\gamma) \circ \bar{h}_m(x)$ for $m \in \mathbb{Z}$ and $\gamma \in \Lambda_{\pm1}$.
Moreover, by the relation $aca^{-1} = c^{k^2}$, $\bar{h}_m \circ \sigma(a^{-1}) (x) = \rho(a^{-1}) \circ \bar{h}_{k^2m}(x)$.
So $\bar{h} \circ \sigma(\gamma) (x) = \rho(\gamma) \circ \bar{h}(x)$ for any $\gamma \in \Gamma$.
\end{proof}

To use Proposition $\ref{Main_1}$ for the proof of the main theorem, we will prove the persistence of the global fixed point $p^0$:
\begin{prop}\label{PGFP}
There is a neighborhood $\mathcal{U}$ of $\rho^0 \in \mathrm{Hom}(\Gamma, \mathrm{Diff}(S^{2n+1}))$ where $\mathrm{Diff}(S^{2n+1})$ is equipped with $C^1$-topology, and a continuous map $\varphi: \mathcal{U} \to S^{2n+1}$ such that $\varphi(\rho^0) = p^0$ and that $\varphi(\rho)$ is a global fixed point of $\rho \in \mathcal{U}$.
\end{prop}
\begin{proof}
\personalnote{Consider the perturbation of the natural homomorphism from $\Gamma$ into $\mathrm{Aff}(\R^{2n+1})$.}
Fix constants $k^{-1} < \lambda < 1$, $0 < \epsilon < k^{-2}$.
Put $\Lambda_{+1} = \{ b_1, \dots, b_{2n}, c\}$.
For $\gamma \in \Lambda_{+1}$, define $m_\gamma \in \mathbb{Z}$ by $m_\gamma = k$ if $\gamma \neq c$ and $m_\gamma = k^2$ if $\gamma = c$ so that $a\gamma a^{-1} = \gamma^{m_\gamma}$.
Put $A_\gamma = d(\phi^0 \circ \rho^0(\gamma)\circ(\phi^0)^{-1})_\mathrm{O}$.
By Lemma \ref{std2}, 
\[
A_\gamma = \left(
\begin{array}{cc}
I_{2n}& u_\gamma\\
0&1
\end{array} \right)
\]
for some $ u_\gamma \in \R^{2n}$.
Let $\|\cdot\|$ be the norm on $\R^{2n+1}= \{x = (x_1, \dots, x_{2n+1}) \mid x_i \in\R\}$ defined by $\|x\| = c_1\sum_{i=1}^{2n}|x_i| + c_2 |x_{2n+1}|$, where $c_1, c_2 >0$ are constants satisfying
\[
-m_\gamma^2\|u_\gamma\|_1c_1 +((1 -\epsilon) m_\gamma-\lambda )c_2 > 0
\]
for $\gamma \in \Lambda_{+1}$, where $\|\cdot\|_1$ is the norm on $\R^{2n}= \{x' = (x'_1, \dots, x_{2n}') \mid x'_i \in \R\}$ defined by $\|x'\|_1 = \sum_{i=1}^{2n}|x'_i|$.
Note that $(1 -\epsilon)m_\gamma - \lambda > (1 - k^{-2})k -1 > 0$.
Then $\| \sum_{j =0}^{m_\gamma-1}A_\gamma^j \| > \epsilon m_\gamma +\lambda$ for $\gamma \in \Lambda_{+1}$, where $\| A \| = \max_{\|x\| = 1} \|Ax\|$ for $A \in GL(2n+1, \R)$.
\personalnote{$x = (x', y') \in \R^{2n}\times\R$.
Since $\sum_{j =0}^{m-1}A^jx= mx + y'(\sum_{j =0}^{m-1}j)(u_\gamma, 0)$,
\[
\|\sum_{j =0}^{m-1}A^jx\| = c_1\|mx' + y'(\sum_{j =0}^{m-1}j)u_\gamma\|_1 + c_2m|y'| \geq c_1m\|x'\|_1 - c_1m^2|y'|\|u_\gamma\|_1+c_2m|y'|
\]
\[
= c_1m\|x'\|_1 +(c_2m-c_1m^2\|u_\gamma\|_1)|y'| > c_1(\epsilon m + \lambda)\|x'\|_1 + c_2(\epsilon m + \lambda)|y'| = (\epsilon m + \lambda)\|x\|.
\]
}

There are neighborhood $\mathcal{U}_1$ of $\rho^0 \in \mathrm{Hom}(\Gamma, \mathrm{Diff}(S^{2n+1}))$ and neighborhoods $V_2 \subset V_1$ of $\phi^0(p^0) = \mathrm{O} \in \R^{2n+1}$ satisfying the following conditions:
\begin{enumerate}
\item For $\gamma \in \Lambda_{+1}$, $i =0, 1$, $j = 0, 1, \dots, m_\gamma$, $\phi^0 \circ \rho(a^i\gamma^j) \circ (\phi^0)^{-1}$ are well-defined on $V_1$.
\item $\| d(\phi^0 \circ \rho(a) \circ (\phi^0)^{-1})_x \| < \lambda$ and $\|d(\phi^0 \circ \rho(\gamma^j) \circ (\phi^0)^{-1})_x - A_\gamma^j \| < \epsilon$ for $x \in V_1$ and $j = 1, \dots, m_\gamma$.
\item For $\gamma \in \Lambda_{+1}$, $\phi^0 \circ \rho(\gamma) \circ (\phi^0)^{-1}(V_2) \subset V_1$.
\end{enumerate}
Using the persistence of a contracting fixed point of a diffeomorphism, we obtain a neighborhood $\mathcal{U}_2 \subset \mathcal{U}_1$ of $\rho^0 \in \mathrm{Hom}(\Gamma, \mathrm{Diff}(S^{2n+1}))$ and a continuous map $\varphi: \mathcal{U}_2 \to (\phi^0)^{-1}(V_2)$ such that $\varphi(\rho^0) = p^0$ and that $\varphi(\rho)$ is a contracting fixed point of $\rho(a)$.
Since $\rho^0(\gamma)(p^0) = p^0$ for $\gamma \in \Lambda_{+1}$, there is a neighborhood $\mathcal{U} \subset \mathcal{U}_2$ of $\rho^0 \in \mathrm{Hom}(\Gamma, \mathrm{Diff}(S^{2n+1}))$ such that $\rho(\gamma)(\varphi(\rho)) \in (\phi^0)^{-1}(V_2)$ for $\rho \in \mathcal{U}$.

We will show that $\varphi(\rho)$ is a global fixed point of $\rho$.
Fix $\rho \in \mathcal{U}$ and $\gamma \in \Lambda_{+1}$.
Put $m = m_\gamma$, $A = A_\gamma$, $x^0 = \phi^0(f(\rho))$, $F = \phi^0 \circ \rho(a) \circ (\phi^0)^{-1}$, $G = \phi^0 \circ \rho(\gamma) \circ (\phi^0)^{-1}$, and $y^0 = G(x^0) - x^0$.
It is sufficient to show that $y^0 = 0$.

By the definition of $\mathcal{U}$, $x^0 \in V_2$.
By the condition (iii),  $G(x^0) \in V_1$.
By the condition (ii),
\[
\| F(G(x^0)) - F(x^0)  \| \leq \lambda \| y^0 \|,
\]
\[
\| G^{j+1}(x^0) -G^j(x^0)\| \geq \|A^jy^0\| - \epsilon\|y^0\| \; (j = 0, \dots, m -1).
\]
So
\[
 \sum_{j = 0}^{m-1}\| G^{j+1}(x^0) -G^j(x^0)\| \geq \sum_{j = 0}^{m-1}\|A^jy^0\| - \epsilon m\|y^0\| .
 \]
Using $F \circ G = G^m \circ F$ and $F(x^0) = x^0$,
\[
F(G(x^0)) - F(x^0) = G^m(x^0) - x^0 =  \sum_{j = 0}^{m-1}\left(G^{j+1}(x^0) -G^j(x^0)\right).
\]
It follows that
\begin{align*}
\lambda \| y^0 \| &\geq \| F(G(x^0)) - F(x^0)  \|\\
 &= \left\|\sum_{j = 0}^{m-1}\left(G^{j+1}(x^0) -G^j(x^0)\right) \right\| \geq \left\|\sum_{j =0}^{m-1}A^j y^0\right\| - \epsilon m\|y^0\|.
\end{align*}
Since $\| \sum_{j =0}^{m-1}A^j \| > \epsilon m +\lambda$, we obtain $\|y^0\| = 0$.
\end{proof}

Now we can prove Theorem \ref{MainThm} assuming Proposition \ref{Main_1}.
\begin{proof}[Proof of Theorem \ref{MainThm} from Proposition \ref{Main_1}]
Let $\rho \in \mathrm{Hom}(\Gamma, \mathrm{Diff}(S^{2n+1}))$ be an action close to $\rho^0 \in \mathrm{Hom}(\Gamma, \mathrm{Diff}(S^{2n+1}, \{p^0\}))$.
By Proposition \ref{PGFP}, we may assume $p^0$ is a global fixed point of $\rho$.
Let $\phi^0_*\rho \in \mathrm{Hom}(\Gamma, \mathcal{G})$ be the homomorphism induced by the chart $\phi^0$ around $p^0$.

If $\rho$ and $\rho^0$ are close as elements of $\mathrm{Hom}(\Gamma, \mathrm{Diff}(S^{2n+1}, p^0))$, then the induced homomorphisms $\phi^0_*\rho$ and $\phi^0_*\rho^0$ in $\mathrm{Hom}(\Gamma, \mathcal{G})$ are also close.
By Proposition $\ref{Main_1}$, we obtain a conjugacy between $\phi^0_*\rho$ and $\phi^0_*\sigma$ with $\sigma$ being a standard action of $\Gamma$ on $S^{2n+1}$.
Using Proposition $\ref{LtoG}$, we see that $\rho, \sigma \in \mathrm{Hom}(\Gamma, \mathrm{Diff}(S^{2n+1}, p^0))$ are conjugate.
\end{proof}

\section{Local rigidity of the local actions}\label{LtoF}
By the homomorphism $D_\mathrm{O}:\mathcal{G} \to \mathcal{F}$ defined by the Taylor expansions, the standard homomorphisms in $\mathrm{Hom}(\Gamma, \mathcal{G})$ induce elements in $\mathrm{Hom}(\Gamma, \mathcal{F})$, which will also be called the \textit{standard homomorphisms} from $\Gamma$ into $\mathcal{F}$.
The goal of this section is to show that Proposition $\ref{Main_1}$ can be reduced to the following proposition:
\begin{prop}\label{Main_2}
The family of standard homomorphisms in $\mathrm{Hom}(\Gamma, \mathcal{F})$ is locally rigid, where $\mathcal{F}$ is equipped with $C^3$-topology.
\end{prop}
\begin{proof}[Proof of Proposition \ref{Main_1} from Proposition \ref{Main_2}]
To reduce the notation, the homomorphism in $\mathrm{Hom}(\Gamma, \mathcal{G})$ induced by $\rho^0 \in \mathrm{Hom}(\Gamma, \mathrm{Diff}(S^{2n+1}))$ will also be denoted by $\rho^0$.
It is sufficient to show that homomorphisms sufficiently close to $\rho^0 \in \mathrm{Hom}(\Gamma, \mathcal{G})$ are conjugate to standard ones.

Assume $\rho \in \mathrm{Hom}(\Gamma, \mathcal{G})$ is sufficiently close to $\rho^0$.
As $D_\mathrm{O}\rho \in \mathrm{Hom}(\Gamma, \mathcal{F})$ is also close to $D_\mathrm{O}\rho^0$, by Proposition $\ref{Main_2}$, there are a standard homomorphism $D_\mathrm{O}\sigma \in \mathrm{Hom}(\Gamma, \F)\; (\sigma \in \mathrm{Hom}(\Gamma, \mathcal{G}))$ and a formal transformation $D_\mathrm{O}h\in\mathcal{F}\; (h \in \mathcal{G})$ satisfying
\begin{equation}\label{conjG}
D_\mathrm{O}\rho^0(\gamma) = D_\mathrm{O}h(D_\mathrm{O}\sigma)(\gamma)(D_\mathrm{O}h)^{-1}
\end{equation}
for all $\gamma \in \Gamma$.
In particular,
\[
D_0\rho(a) = D_0(h\sigma(a)h^{-1}).
\]
So, by Proposition $\ref{St1}$, there exists $h' \in \mathcal{G}$ satisfying $D_0h' = I \in \mathcal{F}$ and
\[
\rho(a) = h'h\sigma(a)h^{-1}h'^{-1}.
\]
Thus, replacing $h'h$ by $h$, we may assume the equation (\ref{conjG}) and $\rho(a) = h\sigma(a)h^{-1}$.
By Lemma $\ref{2_to_1}$ below, for $\gamma = b_1, \dots, b_{2n}, c$,
\[
\rho(\gamma) = h\sigma(\gamma) h^{-1}.
\]
So this equation holds for any $\gamma \in \Gamma$.
\end{proof}

\begin{lem}\label{2_to_1}
Let $m \geq 2$ be an integer, and $f \in \mathcal{G}= \mathcal{G}(\R^{2n+1}, \mathrm{O})$ the germ of diffeomorphism defined by the matrix
\[
\left(\begin{array}{cc}
\frac{1}{k}I_{2n}&0\\
0&\frac{1}{k^2}
\end{array}\right) \in GL(2n+1, \R).
\]
If $g_1, g_2 \in \mathcal{G}$ satisfy $D_\mathrm{O}g_1 = D_\mathrm{O}g_2$, $fg_i = g_i^mf$ $(i =1, 2)$, and
\[
(dg_1)_\mathrm{O} = (dg_2)_\mathrm{O} = I + \left(
\begin{array}{cc}
0 & u \\
0 & 0
\end{array}
\right) \in GL(2n+1, \R),
\]
for some $u \in \R^{2n}$, then $g_1 = g_2$.
\end{lem}
\begin{proof}
Fix constants $r_0 \geq 1$ and  $1 < c < k$ satisfying $mc^{r_0 + 1} < k^{r_0 -2}$.
Let $\|\cdot\|$ be the norm on $\R^{2n+1}$ defined by $\|x\| = c_1\sum_{i=1}^{2n}|x_i| + c_2 |x_{2n+1}|$, where $c_1, c_2 >0$ are constants satisfying
\[
c_1(m+1)\|u\|_1 +c_2(1 -c) < 0,
\]
where $\|\cdot\|_1$ is the norm on $\R^{2n}$ defined by $\|x\|_1 = \sum_{i=1}^{2n}|x_i|$.
Then it is easy to see that $\|(dg_1)_\mathrm{O}^i(v)\|, \|(dg_2)_\mathrm{O}(dg_1)_\mathrm{O}^i(v)\| < c\|v\|$ for all $v \in \R^{2n+1}$ and $i = 0, \dots, m$.
\personalnote{$v = (x, y) \in \R^{2n}\times \R$.
When $Av = v + y(w, 0)$ for some $w \in \R^{2n}$,
\[
\|Av\| = c_1\|x\|_1 + (c_2 + c_1\|w\|_1)|y|< cc_1\|x\|_1 + cc_2|y| = c\|v\|
\]
if $c_1\|w\|_1 +c_2(1 -c) < 0$. }
Then there exist representatives of germs $\tilde{g_i} \in g_i, \tilde{f}\in f$ and a constant $R_1 > 0$ satisfying the following condition:
\begin{itemize}
\item $\tilde{g_2}^i\tilde{g_1}^j$ is well-defined on $B_{R_1}$ for $i,j = 1, \dots, m$,
\item $\tilde{f}\tilde{g_i} = \tilde{g_i}^m\tilde{f}$ on $B_{R_1}$,
\item $\| \tilde{g_1}^i(v) \|, \|\tilde{g_2}\tilde{g_1}^i(v)\| \leq c \|v\|$, $\| \tilde{g_2}(v) - \tilde{g_2}(v') \| \leq c \| v - v' \|$ for $v, v' \in B_{R_1}$, $i = 0, \dots, m$,
\end{itemize}
where $B_{R_1} \subset \R^{2n+1}$ is the ball of radius $R_1$ with respect to  $\| \cdot \|$ centered at $\mathrm{O} \in \R^{2n+1}$.
Fix a constant $0 < R_2 < R_1$ with
\[
\tilde{g_1}^i\tilde{f}(v), \tilde{g_2}\tilde{g_1}^i\tilde{f}(v) \in B_{R_1} \; ( v \in B_{R_2}, i = 0, \dots, m).
\]

Since $D_\mathrm{O}g_1 = D_\mathrm{O}g_2$,
\[
\Delta = \sup_{v \in B_{R_2}} \frac{\| \tilde{g_1}(v) - \tilde{g_2}(v) \|}{\| v \|^{r_0}}
\]
is a finite number.
It is sufficient to show that $\Delta = 0$.
For any $v \in B_{R_2}$,
\begin{align*}
\| \tilde{g_1}(v) - \tilde{g_2}(v) \| &= \| \tilde{f}^{-1}\tilde{g_1}^m\tilde{f}(v) - \tilde{f}^{-1}\tilde{g_2}^m\tilde{f}^{-1}(v) \| \\
&\leq k^2 \|\tilde{g_1}^m\tilde{f}(v) - \tilde{g_2}^m\tilde{f}(v) \| \\
&\leq k^2\sum_{i = 1}^m \|\tilde{g_2}^{m - i} \tilde{g_1}^i \tilde{f} (v) - \tilde{g_2}^{m - i + 1}\tilde{g_1}^{m -1}\tilde{f} (v)\| \\
&\leq k^2c\sum_{i = 1}^m \|\tilde{g_1}^i \tilde{f} (v) - \tilde{g_2}\tilde{g_1}^{i -1}\tilde{f} (v)\| \\
&\leq k^2c\Delta \sum_{i = 1}^m \|\tilde{g_1}^i \tilde{f} (v)\|^{r_0} \\
&\leq k^2c\Delta m \left( c\frac{1}{k}\|v\|\right)^{r_0} \\
&= \Delta \frac{mc^{r_0 +1}}{k^{r_0 - 2}} \|v\|^{r_0}.
\end{align*}
It follows that
\[
\frac{\| \tilde{g_1}(v) - \tilde{g_2}(v) \|}{\| v \|^{r_0}} \leq \frac{mc^{r_0 +1}}{k^{r_0 - 2}}\Delta.
\]
Taking the supremum for $v \in B_{R_2}$,
\[
\Delta \leq \frac{mc^{r_0 +1}}{k^{r_0 - 2}}\Delta.
\]
As $\frac{mc^{r_0 +1}}{k^{r_0 - 2}} < 1$, we have $\Delta = 0$.
\end{proof}

\section{Local rigidity of homomorphisms into the group of formal transformations}\label{F}
In this section, we will give a proof of Proposition \ref{Main_2}.
First, we will show that, while the group $\F$ is an infinite-dimensional Lie group, Proposition \ref{Main_2} can be reduced to a problem (Proposition \ref{Main_3}) of homomorphisms into a finite-dimensional Lie group $\P_r$.
The finite-dimensional Lie group $\P_r $ will be defined in the next subsection.
Next, we will show that the proof of Proposition \ref{Main_3} can be divided into three steps.
Each steps will be discussed in subsections \ref{exthomtoLie}, \ref{persistcycl}, and \ref{persistnilp}, respectively.

\subsection{The group of jets of diffeomorphisms at a point}\label{jets}
The group $\P_r\;(r \geq 0)$ is defined as an analogy of the group $J_r(\R^m, \mathrm{O})$ of $(r+1)$-th jets of the diffeomorphisms defined around $\mathrm{O} \in\R^m$ that fix $\mathrm{O} \in \R^m$.
Note that $J_0(\R^m, \mathrm{O}) = GL(m, \R)$ and that the Lie group $\P_0$ was already defined in Section \ref{Stb}
To motivate our definition of the group $\P_r\;(r \geq 0)$, we begin with an observation on $J_r(\R^m, \mathrm{O})$.
Let $\mathfrak{X}(\R^m)$ be the Lie algebra of the smooth vector fields on $\R^m$, and $\mathrm{Poly}(\R^m, \mathrm{O}) \subset \mathfrak{X}(\R^m)$ the subalgebra of the polynomial vector fields that vanish at $\mathrm{O} \in \R^m$.
We will observe that $J_r(\R^m, \mathrm{O})$ arises naturally from a gradation on $\mathrm{Poly}(\R^m, \mathrm{O})$.

Let $\mathrm{Poly}^{(r)}(\R^m, \mathrm{O}) \subset \mathrm{Poly}(\R^m, \mathrm{O})$ be the subspace of vector fields whose coefficients are homogeneous polynomials of degree $r+1$.
The Lie algebra $\mathrm{Poly}(\R^m, \mathrm{O})$ has a \textit{gradation}
\[
\mathrm{Poly}(\R^m, \mathrm{O}) = \bigoplus_{r\geq0}\mathrm{Poly}^{(r)}(\R^m, \mathrm{O}),
\]
in the sense that $[\mathrm{Poly}^{(r)}(\R^m, \mathrm{O}), \mathrm{Poly}^{(r')}(\R^m, \mathrm{O})] \subset \mathrm{Poly}^{(r + r')}(\R^m, \mathrm{O})$.
Note that $\mathrm{Poly}^{(0)}(\R^m, \mathrm{O}) = \mathfrak{gl}(m, \R)$.
For $r \geq 0$, the subspace
\[
\mathfrak{j}_r = \bigoplus_{0\leq q \leq r}\mathrm{Poly}^{(q)}(\R^m, \mathrm{O})
\]
can be equipped with the Lie bracket $[\,\cdot, \cdot\,]_{\mathfrak{j}_r}$ defined by the following condition:
For $X \in \mathrm{Poly}^{(p)}(\R^m, \mathrm{O})$, $Y \in \mathrm{Poly}^{(q)}(\R^m, \mathrm{O})$,
\[
[X, Y]_{\mathfrak{j}_r} = 
\begin{cases}
[X, Y]_{\mathfrak{X}(\R^m)} & (\text{if }p +q \leq r) \\
0 & (\text{otherwise})
\end{cases}.
\]
Note that $\mathfrak{j}_r \subset \mathrm{Poly}(\R^m, \mathrm{O})$, $r \geq 1$ is not a subalgebra.
Let
\[
\mathfrak{n}_r = \bigoplus_{1\leq q \leq r}\mathrm{Poly}^{(q)}(\R^m, \mathrm{O}) \subset \mathfrak{j}_r
\]
be a nilpotent Lie subalgebra of $\mathfrak{j}_r$, and $N_r$ the connected simply-connected nilpotent Lie group with its Lie algebra $\mathfrak{n}_r$. 
The group $GL(m, \R)$ acts naturally on $\mathfrak{n}_r$ by $GL(m, \R)\times \mathfrak{n}_r \to \mathfrak{n}_r, \;(g, X) \mapsto g_*X$, where $g_*X$ is the push-forward of $X \in \mathfrak{X}(\R^m)$ by the diffeomorphism $g\in GL(m, \R) \subset \mathrm{Diff}(\R^m)$.
Then the diffeomorphism $\exp:\mathfrak{n}_r \to N_r$ induces an action of $GL(m, \R)$ on $N_r$.
The semidirect product $GL(m, \R) \ltimes N_r$ is the group $J_r(\R^m, \mathrm{O})$ with its Lie algebra $\mathfrak{j}_r$.

To define the group $\P_r = \P_r(\R^{2n+1}, \mathrm{O})$, we consider a certain gradation of $\mathrm{Poly}(\R^{2n+1}, \mathrm{O})$ associated with the matrix
\[
\phi^0_*(\rho^0(a)) =
\left( \begin{array}{cc}
\frac{1}{k}I_{2n}&0\\
0&\frac{1}{k^2}
\end{array} \right) \in GL(2n+1, \R).
\]
In fact,
\[
\mathrm{Poly}(\R^{2n+1}, \mathrm{O}) = \bigoplus_{r \geq -1} \p^{(r)}(\R^{2n+1}, \mathrm{O})
\]
is a gradation, where
\[
\p^{(r)}(\R^{2n+1}, \mathrm{O})= \{X \in \mathrm{Poly}(\R^{2n+1}, \mathrm{O})  \mid \phi^0_*(\rho^0(a))_*X = k^rX \}.
\]
For $r \geq 0$, in the same way as $\mathfrak{j}_r$, the subspace
\[
\p_r = \bigoplus_{0 \leq q \leq r} \p^{(q)}(\R^{2n+1}, \mathrm{O})
\]
can be equipped with a Lie bracket.
Note that the group $\P_0\subset \mathrm{Diff}(\R^{2n+1}, \mathrm{O})$ defined in Section \ref{Stb} is a Lie group with its Lie algebra $\p_0\subset \mathfrak{X}(\R^{2n+1}, \mathrm{O})$.
Let $\mathcal{Q}_r$ be the connected simply-connected nilpotent Lie group with its Lie algebra 
\[
\mathfrak{q}_r = \bigoplus_{1 \leq q \leq r} \p^{(q)},
\]
on which $\P_0$ acts naturally in the same way as that of $GL(m, \R)$ on $N_r$.
Finally, we obtain  the semidirect product $\P_0 \ltimes \mathcal{Q}_r$, which will be denoted by $\P_r$.

Recall that each $F \in \F$ admits a decomposition $F(x) = \sum_{r\geq -1} F^{(r)}(x)$ as in Section \ref{Stb}.
It is straightforward to see that there is a natural identification of $\P_r$ with the subspace
\[
\{F \in \F \mid F^{(q)} = 0 \text{ for } q = -1,\; q \geq r +1\} \subset \F.
\]
There are natural surjective group homomorphisms $\mathcal{F} \to \P_r$ and $\P_q \to \P_r, \,(q > r)$ defined by forgetting higher-order terms.
By abuse of notation, such homomorphisms will be denoted by  $\pi_r:\mathcal{F} \to \P_r$, $\pi_r:\P_q \to \P_r$.

To reduce the notation, the homomorphism $D_\mathrm{O}\phi^0_*\rho^0 \in \mathrm{Hom}(\Gamma, \F)$ induced by the action $\rho^0 \in \mathrm{Hom}(\Gamma, \mathrm{Diff}(S^{2n+1}))$ will be denoted by $R^0$.
We will show that Proposition \ref{Main_2} follows from the next proposition, which claims that the family of the standard homomorphisms in $\mathrm{Hom}(\Gamma, \P_r)$ $(r \geq 2)$ is locally rigid if $\P_r$ is equipped with the topology induced from $\P_1$:
\begin{prop}\label{Main_3}
For $f \in \mathrm{Hom}(\Gamma, \P_r)$, $r \geq 2$ with $\pi_1f \in \mathrm{Hom}(\Gamma, \P_1) $ sufficiently close to $\pi_1R^0$, there exists a standard homomorphism $S \in \mathrm{Hom}(\Gamma, \F)$ such that $\pi_rf, \pi_rS \in \mathrm{Hom}(\Gamma, \P_r)$ are conjugate.
\end{prop}

\begin{proof}[Proof Proposition \ref{Main_2} from Proposition \ref{Main_3}]
If a homomorphism $R:\Gamma \to \mathcal{F}$ is $C^3$-close to $R^0$, then $\pi_1R, \pi_1R^0 \in \mathrm{Hom}(\Gamma, \P_1)$ are also close.
By Proposition \ref{Main_3}, we obtain $h \in \P_1$ and a standard homomorphism $S \in \mathrm{Hom}(\Gamma, \F)$ such that $h(\pi_rR(\gamma))h^{-1} = \pi_rS(\gamma)$ for $\gamma \in \Gamma$.
In particular, $h(\pi_1R(\gamma))h^{-1} = \pi_1S(\gamma)$.

Fix $H \in \mathcal{F}$ with $\pi_1H = h \in \P_1$.
Then $\pi_1(HR(a)H^{-1}) = \pi_1(S(a))$.
Thus by Proposition \ref{St2}, replacing $H$ if necessary, we may assume that $HR(a)H^{-1} = S(a)$.
Then Lemma \ref{Main2_1} below shows that $HR(\gamma)H^{-1} = S(\gamma)$ for all $\gamma \in \Gamma$.
\end{proof}

\begin{lem}\label{Main2_1}
Let $S \in \mathrm{Hom}(\Gamma, \F)$ be a standard homomorphism.
Assume $R \in \mathrm{Hom}(\Gamma, \mathcal{F})$ satisfies the following condition:
\begin{itemize}
\item $R(a) = S(a)$,
\item $\pi_1R(\gamma) = \pi_1S(\gamma)$ for $\gamma \in \Gamma$.
\end{itemize}
Then $R = S$.
\end{lem}
\begin{proof}
Set $F = R(b_i)$, and $G= R(a) = S(a)$.
Since $ab_i = b_i^ka$ and $G \in \P_0$, we see that 
\[
G \circ F^{(r)} =(F^k)^{(r)} \circ G.
\]
We will show that $F$ is uniquely determined by this equation and $F^{(r)},\,r = 0, 1$.
For $r = -1$, we can show that $F^{(-1)} = 0$.
In fact, since $F^{(0)} = (S(b_i)) ^{(0)} = \mathrm{id} \in \P_0$, we see that $(F^k)^{(-1)} = kF^{(-1)}$.
By $G \circ F^{(-1)} \circ G^{-1} = k^{-1}F^{(-1)}$,
\[
k^{-1}F^{(-1)} = kF^{(-1)}
\]
and thus $F^{(-1)} = 0$.
For $r \geq 2$, by $F^{(-1)} = 0$, $F^{(0)} = \mathrm{id}$, we see that
\[
(F^k)^{(r)} = kF^{(r)} + \Phi,
\]
where $\Phi$ is a map determined by  $F^{(r')}\;(r' < r)$.
Using $G \circ F^{(r)} \circ G^{-1} = k^rF^{(r)}$,we see that $F^{(r)}$ is determined by $F^{(r')}\;(r' < r)$.
Thus $F = R(b_i) = S(b_i)$.

It remains to show that $R(c) = S(c)$.
Fix $b_i, b_j$ with $[b_i, b_j] = c^{m}$, $m \neq 0$.
As we have seen that $R(b_i) = S(b_i)$ and $R(b_j) = S(b_j)$, we obtain $R(c)^{m} = S(c)^{m}$.
Since $S(c)^{(0)} = \mathrm{id}$ and $S(c)^{(1)} = 0$,
\[
mR(c)^{(2)}  =(R(c)^{m})^{(2)} = (S(c)^{m})^{(2)} = mS(c)^{(2)}.
\]
Thus $R(c)^{(2)} = S(c)^{(2)}$.
Since $ac = c^{k^2}a$, by a similar argument to that of $F = R(b_i)$, we see that $F = R(c)$ is determined by $F^{(r)},\,r= 0, 1, 2$.
It follows that $R(c) =S(c)$.
\end{proof}

The proof of Proposition \ref{Main_3} can be divided into three steps as follows.
Recall that $\Gamma$ is a subgroup of $AN$ in the Iwasawa decomposition $SU(n+1, 1) = KAN$ defined in Section \ref{std}, and that the closure of $\Gamma$ in $AN$ is $\langle a \rangle \ltimes N$.
As the first step of the proof, we will show that homomorphisms of $\Gamma$ can be extended to  $\langle a \rangle \ltimes N$:
\begin{prop}\label{extf}\personalnote{This is also true for $r\geq 0$.}
If $f \in \mathrm{Hom}(\Gamma, \P_r)$, $r \geq 2$ is a homomorphism with $\pi_1f \in \mathrm{Hom}(\Gamma, \P_1)$ sufficiently close to $\pi_1R^0$, then $f$ can be uniquely extended to a continuous homomorphism $\bar{f}:\langle a \rangle \ltimes N \to \P_r$.
Furthermore, if $f(a) \in \P_0$, then $\bar{f}(N) \subset \mathcal{Q}_r$.
\end{prop}

As the next step, using the extension $\bar{f}$, we will show the persistence of $f(a)$:
\begin{prop}\label{persistcyclprop}
If $f \in \mathrm{Hom}(\Gamma, \P_r)$, $r \geq 2$ is a homomorphism with $\pi_1f \in \mathrm{Hom}(\Gamma, \P_1)$ sufficiently close to $\pi_1R^0$ and $f(a) \in \P_0 \subset \P_r$, then $f(a) = \pi_rR^0(a)  \in \P_0$.
\end{prop}

The final step of the proof is to prove the persistence of $f(N)$:
\begin{prop}\label{persistnilpprop}
If $f \in \mathrm{Hom}(\Gamma, \P_r)$, $r \geq 2$ is a homomorphism with $\pi_1f \in \mathrm{Hom}(\Gamma, \P_1)$ sufficiently close to $\pi_1R^0$ and $f(a)  = \pi_rR^0(a)  \in \P_0 \subset \P_r$, then there is an element $h \in \P_0$ such that $h\bar{f}(N)h^{-1} = \pi_rR^0(N)$, where $\bar{f}$ is the extension of $f$ obtained by Proposition \ref{extf}.
\end{prop}

The proof of the above propositions will be given in the remaining subsections.
Proposition \ref{Main_3} can be derived from these three propositions as follows:
\begin{proof}[Proof of Proposition \ref{Main_3} from Proposition \ref{extf}, \ref{persistcyclprop}, and \ref{persistnilpprop}]
Consider a homomorphism $f \in \mathrm{Hom}(\Gamma, \P_r)$, $r \geq 2$ with $\pi_1f \in \mathrm{Hom}(\Gamma, \P_1)$ sufficiently close to $\pi_1R^0$.
By Proposition \ref{St2}, after replacing $f$ with its conjugate, we may assume that $f(a) \in \P_0 \subset \P_r$.
We obtain the extension $\bar{f} \in \mathrm{Hom}(\langle a \rangle \ltimes N, \P_r)$ by  Proposition \ref{extf}.
By Proposition \ref{persistcyclprop}, $f(a) = \pi_rR^0(a)$.
After replacing $f$ with its conjugate, by Proposition \ref{persistnilpprop}, we may assume that $\bar{f}(N) = \pi_rR^0(N)$.
Using Lemma \ref{closurerigid} below, we obtain a standard homomorphism $S \in \mathrm{Hom}(\Gamma, \F)$ such that $\pi_r f = \pi_r S \in \mathrm{Hom}(\Gamma, \P_r)$.
\end{proof}
\begin{lem}\label{closurerigid}
Let $f \in \mathrm{Hom}(\Gamma, \P_r)$ $(r \geq 2)$ be a homomorphism with $\pi_1f:\Gamma \to \P_1$ sufficiently close to $\pi_1R^0$ and $\bar{f}:\langle a \rangle \ltimes N \to \P_r$ the extension of $f$.
If
\begin{itemize}
\item $f(a) =\pi_rR^0(a)$, and
\item $\bar{f}(N) = \pi_rR^0(N)$,
\end{itemize}
then there is a standard homomorphism $S \in \mathrm{Hom}(\Gamma, \F)$ such that $f = \pi_r S$.
\end{lem}
\begin{proof}
By the definition of the standard homomorphism, it is sufficient to show that there is an embedding $\iota:\Gamma \to AN$ as a standard subgroup such that $f = \pi_rR^0 \circ \iota$.

Since the action $\rho^0$ of $AN$ on $S^{2n+1}$ admits the same global fixed  point $p^0$ as that of $\Gamma$, $R^0 = D_\mathrm{O}\phi^0_*\rho^0 \in \mathrm{Hom}(\Gamma, \F)$ admits the natural extension to $D_\mathrm{O}\phi^0_*\rho^0 \in \mathrm{Hom}(AN, \F)$, which will also be denoted by $R^0 \in \mathrm{Hom}(AN, \F)$.
Note that $\pi_rR^0:AN \to \P_r$ is an automorphism onto its image.

By the assumption, we obtain the embedding $\iota = (\pi_rR^0)^{-1} \circ f:\Gamma \to AN$ satisfying $\iota(a) = a$ and $\iota(\Lambda) \subset N$.
It remains to show that $\iota(\Lambda) \subset N$ is a lattice.
Since $\pi_1f$ is close to $\pi_1R^0$, $\iota$ is close to the inclusion $\Gamma \subset AN$.
In particular, $\iota|_\Lambda$ is close to the inclusion $\Lambda \subset N$ in $\mathrm{Hom}(\Lambda, N)$.
It follows that $\iota(\Lambda) \subset N$ is a lattice.
\end{proof}
\subsection{Extension of homomorphisms to a Lie group}\label{exthomtoLie}
In this subsection, we will give a proof of Proposition \ref{extf}.
Let $f \in \mathrm{Hom}(\Gamma, \P_r)$ be a map such that $\pi_1f$ is close to $\pi_1R^0$.
First, by Proposition \ref{St2}, after replacing $f$ with its conjugate, we may assume that $f(a) \in \P_0 \subset \P_r$.
Next, we will show $\pi_0f(b_i) = \mathrm{id} \in \P_0$:
\begin{prop}
If $g, h \in \P_0$ satisfying $g h = h^kg$ are close to $\pi_0R^0(a), \mathrm{id} \in  \P_0$, respectively, then $h = \mathrm{id}$.
\end{prop}
\begin{proof}
To prove this proposition, we will use Weil's implicit function theorem \cite{Weil}:
Let $F_1:(M_0, x_0) \to (M_1, x_1)$, $F_2:(M_1, x_1) \to (M_2, x_2)$ be smooth maps between  smooth manifolds.
If $F_2 \circ F_1$ is the constant map at $x_2$ and
\[
\mathrm{Im}(dF_1)_{x_1} = \mathrm{Ker}(dF_2)_{x_2},
\]
then there is a neighborhood $U$ of $x_0 \in M_0$ such that $F_1(U)$ is a neighborhood of $x_1 \in F_2^{-1}(x_2)$.

Put
\begin{align*}
(M_0, x_0) &= (\P_0, \pi_rR^0(a)), \\
(M_1, x_1) &= (\P_0^2, (\pi_rR^0(a), \mathrm{id})), \\
(M_2, x_2) &= (\P_0, \mathrm{id}),
\end{align*}
and 
\[
F_1(g) = (g, \mathrm{id}),\; F_2(g, h) = g hg^{-1}h^{-k}.
\]
To apply Weil's Implicit function theorem to this setting, it is sufficient to show that $\mathrm{Im}(dF_1)_{x_1} = \mathrm{Ker}(dF_2)_{x_2}$.
Then the claim follows immediately.

Identifying the tangent space $T_{\pi_rR^0(a)}\P_0$ at $\pi_rR^0(a) \in \P_0$ with $T_{\mathrm{id}}\P_0 = \p_0$ by the left translation, it is straightforward to see that
\[
(dF_1)_{x_1}(X) = (X, 0),\; (dF_2)_{x_2}(X, Y) = \mathrm{Ad}(\pi_rR^0(a))(Y) -kY
\]
for $X,Y \in \p_0$.
By $\mathrm{Ad}(\pi_rR^0(a)) = \mathrm{id} \in \mathrm{Aut}(\p_0)$, it follows that $\mathrm{Ker}(dF_2)_{x_2} = \{(X, 0)\}$.
\end{proof}

Recall that $\P_r = \P_0 \ltimes \mathcal{Q}_r$.
By the above proposition, we may assume that $f(\Lambda) \subset \P_r$ is contained in $\mathcal{Q}_r$.
By definition, $\Lambda$ is a lattice of the nilpotent Lie group $N$.
We will use the following fundamental result on lattices in nilpotent Lie groups:
\begin{thm}[\cite{Raghunathan}]\label{nilpext}
Let $N$ and $V$ be two nilpotent simply connected Lie groups, and $H$ a uniform subgroup of $N$.
Then any continuous homomorphism $f:H\to V$ can be extended uniquely to a continuous homomorphism $\bar{f}:N \to V$.
\end{thm}

Thus $f:\Lambda \to \mathcal{Q}_r$ can be uniquely extended to a continuous homomorphisms $\bar{f}:N \to \mathcal{Q}_r$.
Note that $f(a)\bar{f}(g)f(a)^{-1} = \bar{f}(aga^{-1})$ for all $g \in N$, since both sides of the equation are continuous extensions of the homomorphism $\lambda \in \Lambda \mapsto f(a)f(\lambda)f(a)^{-1} = f(a\lambda a^{-1}) \in \mathcal{Q}_r$.
Thus we proved Proposition \ref{extf}.

\subsection{Persistence of the cyclic subgroup}\label{persistcycl}
In this subsection, we will prove Proposition \ref{persistcyclprop}.
Let $f \in \mathrm{Hom}(\Gamma, \P_r)$, $r \geq 2$ be a homomorphism with $\pi_1f \in \mathrm{Hom}(\Gamma, \P_1)$ sufficiently close to $\pi_1R^0$ and $f(a) \in \P_0 \subset \P_r$.
To prove $f(a) = \pi_rR^0(a)  \in \P_0$, we will construct a Lie subalgebra of $\mathfrak{X}(\R^{2n+1})$ on which $f(a), \pi_rR^0(a) \in \P_0 \subset \mathrm{Diff}(\R^{2n+1})$ act in the same way.
The subalgebra of  $\mathfrak{X}(\R^{2n+1})$ is , using the terminology in Section \ref{Heis}, a Heisenberg connection.

Let $\bar{f} \in \mathrm{Hom}(\langle a \rangle \ltimes N, \P_r)$ be the extension of $f$ obtained by Proposition \ref{extf} and $\bar{f}_*:\n \to \mathfrak{q}_r$ the homomorphism of Lie algebras induced by the homomorphism of Lie group $\bar{f}|_N:N \to \mathcal{Q}_r$.
As the linear subspace $\mathfrak{q}_r \subset \mathfrak{X}(\R^{2n+1}, \mathrm{O})$ is not a Lie subalgebra, the linear map $\bar{f}_*:\n \to \mathfrak{X}(\R^{2n+1}, \mathrm{O})$ induced by the inclusion $\mathfrak{q}_r \subset \mathfrak{X}(\R^{2n+1}, \mathrm{O})$ is not necessarily a homomorphism of Lie algebras.
But we can construct a homomorphism from $\n$ into $\mathfrak{X}(\R^{2n+1}, \mathrm{O})$ in the following way.
The $\p^{(s)}(\R^{2n+1}, \mathrm{O})$-component of $X \in \p_r$ will be denoted by $X^{(s)}$.
Recall that $X_1, \dots, X_{2n}, Y \in \n$ is the basis of $\n$ with the relations $[X_i, X_j] = m_{ij}Y$, $[X_i, Y] = 0$ satisfying $b_i = \exp X_i$, $c = \exp Y$ as described in Section \ref{std}.

\begin{lem}\label{Hconnconstruct}
For $r \geq 3$, let $\varphi:\n \to \mathfrak{X}(\R^{2n+1}, \mathrm{O})$ be the linear map defined by 
\[
\varphi(X_i) = \bar{f}_*(X_i)^{(1)}\;(i = 1, \dots, 2n), \;\varphi(Y) = \bar{f}_*(Y)^{(2)}.
\]
Then $\varphi$ is a homomorphism of Lie algebras.
\end{lem}
\begin{proof}
As $\p_r = \bigoplus_{s\leq r}\p^{(s)}(\R^{2n+1}, \mathrm{O})$ is a gradation, and $\bar{f}_*:\n \to \mathfrak{q}_r$ is a homomorphism of Lie algebras, by the relations on $\n$, we see that
\begin{align*}
0 &= \bar{f}_*(Y)^{(1)},\\
[\bar{f}_*(X_i)^{(1)}, \bar{f}_*(X_j)^{(1)}] &= m_{ij}\bar{f}_*(Y)^{(2)}, \\
[\bar{f}_*(X_i)^{(1)}, \bar{f}_*(Y)^{(2)}] &= 0.
\end{align*}
It follows that $\varphi$ is a homomorphism of Lie algebras.
\end{proof}

Observe that both $f(a), \pi_rR^0(a) \in \P_0 \subset \mathrm{Diff}(\R^{2n+1})$ preserve $\varphi(\n)\subset \mathfrak{X}(\R^{2n+1})$ and the restrictions on it coincide.
In fact, by the relations $ab_ia^{-1} = b_i^k$, we see that $f(a)_*\bar{f}_*(X_i) = \mathrm{Ad}(f(a))\bar{f}_*(X_i) = k\bar{f}_*(X_i) \in \p_r$ and thus $f(a)_*\bar{f}_*(X_i)^{(1)} = k\bar{f}_*(X_i)^{(1)}$.
Since $\bar{f}_*(X_i)^{(1)} \in \p^{(1)}$, $\pi_rR^0(a)_*(\bar{f}_*(X_i)^{(1)}) = k\bar{f}_*(X_i)^{(1)}$.

\begin{proof}[Proof of Proposition \ref{persistcycl}]
We will consider the action of the element $g = f(a)^{-1}\,\pi_rR^0(a)\in \P_0 \subset \mathrm{Diff}(\R^{2n+1})$ on $\varphi(\n)\subset \mathfrak{X}(\R^{2n+1})$.
By Lemma \ref{stdalg}, $\phi^0_* \circ \rho^0_*(\n) \subset \mathfrak{X}(\R^{2n+1})$ is a frame field on $\R^{2n+1}\setminus\{\mathrm{O}\}$.
Thus if $\pi_1f$ is sufficiently close to $\pi_1R$, by Lemma \ref{Hconnconstruct}, $\varphi(\n)$ is a Heisenberg connection on a neighborhood of a point $x \in \R^{2n+1}\setminus \{\mathrm{O} \}$.

If $\pi_1f$ is sufficiently close to $\pi_1R$, we may assume that there are neighborhoods $V \subset U$ of $x \in \R^{2n+1}$ such that
\begin{itemize}
\item $\varphi(\n)|_U$ is a Heisenberg connection on $U$, and
\item $g(V) \subset U$.
\end{itemize}
By the above observation, $g$ preserves the Heisenberg connection $\varphi(\n)$.
By Lemma \ref{Hpres} (ii), there is a neighborhood $V' \subset V$ of $x$ and $X \in Z(\varphi(\n))$ such that $g|_{W} = \exp(X)|_{V'}$.

Put $E^0 = \phi^0_* \circ \rho^0_*(E)$.
Note that $g_*E^0 = 0$.
It follows that $[X, E^0] = 0$.
Since $[E^0, X] = rY$ for all $Y \in \p^{(r)}$, $E^0$ is a dilation of $\varphi(\n)$ by $e$.

By Lemma \ref{Hdil} there is a dilation $E$ of $\varphi(\n)$ by $e$ which is also a dilation of $Z(\varphi(\n))$ by $e$.
Since $E^0$ and $E$ are dilations of $\varphi(\n)$ by $e$, we obtain $Y \in Z(\varphi(\n))$ such that $E^0 = E +Y$.
Thus $[X, E + Y] = 0$.
Since $X, Y \in Z(\varphi(\n))$ and $E$ is a dilation of $ Z(\varphi(\n))$, we see that $X = 0$.
Thus $g$ is the identity on $W$.
So $g = \mathrm{id} \in \P_0$.
\end{proof}

\subsection{Persistence of the nilpotent Lie group}\label{persistnilp}
In this section, we will prove Proposition \ref{persistnilpprop}.
As we will see it later, Proposition \ref{persistnilpprop} follows from Proposition \ref{nplushom} below, which claims certain local rigidity of a homomorphism of Lie algebras.
Let us set up notation and terminology to state and prove Proposition \ref{nplushom}.

Consider the homomorphism of Lie algebras $\phi^0_*\circ \rho^0_*:\mathfrak{g} = \mathfrak{su}(n+1, 1) \to \mathfrak{X}(\R^{2n+1})$ defined in Subsection \ref{stdsu}.
For simplicity of notation, put $\iota^0 = \phi^0_*\circ \rho^0_*$.
Recall that $\mathfrak{g}$ admits the gradation $\mathfrak{g} = \bigoplus_{|\lambda| \leq 2, \lambda \in \mathbb{Z}} \mathfrak{g}^{(\lambda)}$.
The graded Lie algebra $\mathfrak{g}$ admits graded Lie subalgebras $\mathfrak{n}^{+} = \mathfrak{g}^{(+1)} \oplus \mathfrak{g}^{(+2)}$ and $\mathfrak{n}^{-} = \mathfrak{g}^{(-1)} \oplus \mathfrak{g}^{(-2)}$.
Note that $\mathfrak{n}^{+}$ is the Lie algebra of $N$, which was denoted by $\mathfrak{n}$.

By the definition of  $\rho^0$ and $\phi^0$, it is easy to see that the image of $\iota^0:\mathfrak{g} \to \mathfrak{X}(\R^{2n+1})$ is contained in the subalgebra $\p \subset \mathfrak{X}(\R^{2n+1}) $ of polynomial vector fields on $\R^{2n+1}$.
Moreover, $\p$ admits a gradation which is compatible with that of $\mathfrak{g}$:
Consider the gradation $\p = \bigoplus_{r \geq -2} \p^{(r)}$, where
\[
\p^{(r)} = \p^{(r)}(\R^{2n+1})= \{X \in\p  \mid \phi^0_*(\rho^0(a))_*X = k^rX \}.
\]
In general, for any graded Lie algebras $\mathfrak{g} = \bigoplus_r \mathfrak{g}^{(r)},\; \mathfrak{h} =\bigoplus_r \mathfrak{h}^{(r)}$, a homomorphism of Lie algebras $f :\mathfrak{g} \to \mathfrak{h}$ is said to be a \textit{homomorphism of graded Lie algebras} if
\[
f(\mathfrak{g}^{(r)}) \subset \mathfrak{h}^{(r)}
\]
for all $r$.
The space of homomorphisms of graded Lie algebras will be denoted by $\mathrm{Hom}_\mathrm{gr}(\mathfrak{g}, \mathfrak{h})$.
By the definition of $\p^{(r)}$,  $\iota^0:\mathfrak{g} \to \p$ is a  homomorphism of graded Lie algebras.

Recall that the group $\P_0$, which is a subgroup of $\mathrm{Diff}(\R^{2n+1})$, acts naturally on $\mathfrak{X}(\R^{2n+1})$.
Note that $\P_0$ preserves $\p \subset \mathfrak{X}(\R^{2n+1})$ and that $\P_0$ acts on $\p$ by homomorphisms of graded Lie algebras.

We can now rephrase Proposition \ref{persistnilpprop} as the following proposition, which is, so to speak, a local rigidity of a homomorphism of graded Lie algebra.
\begin{prop}\label{nplushom}
If $\iota \in \mathrm{Hom}_\mathrm{gr}(\mathfrak{n}^{+}, \p)$ is a homomorphism sufficiently close to $\iota^0 = \iota^0|_{\mathfrak{n}^{+}} \in \mathrm{Hom}_\mathrm{gr}(\mathfrak{n}^{+}, \p)$, then there exists $h \in \P_0$ such that
\[
\iota(\mathfrak{n}^{+}) = h_* \iota^0(\mathfrak{n}^{+}).
\]
\end{prop}
\begin{proof}[Proof of Proposition \ref{persistnilpprop} from Proposition \ref{nplushom}]
Let $f \in \mathrm{Hom}(\Gamma, \P_r)$, $r \geq 2$ be a homomorphism with $\pi_1f \in \mathrm{Hom}(\Gamma, \P_1)$ sufficiently close to $\pi_1R^0$ and $f(a)  = \pi_rR^0(a)  \in \P_0 \subset \P_r$, and  $\bar{f} \in \mathrm{Hom}(\langle a \rangle \ltimes N, \P_r)$ the extension of $f$ obtained by Proposition \ref{extf}.
Our goal is to find an element $h \in \P_0$ satisfying $h\bar{f}(N)h^{-1} = \pi_rR^0(N)$.

Since $f(a) = \pi_r R^0(a)$, we see that $\bar{f}_*:\n^+ \to \p_r$ is a homomorphism of graded Lie algebras and that the map $\bar{f}_*:\n^+ \to \p$ induced by the inclusion $\p_r \subset \p$ is also a homomorphism of graded Lie algebras.
Applying Proposition \ref{nplushom} to $\iota = \bar{f}_*:\n^+ \to \p$, we obtain $h \in \P_0$ such that $\bar{f}_*(\mathfrak{n}^+) = h_* \iota^0(\mathfrak{n}^+)$.
Replacing $f$ by $h^{-1}fh$, we may assume $\bar{f}_*(\mathfrak{n}^+) =  \iota^0(\mathfrak{n}^+)$.
As $R^0_* = \iota^0:\n^+ \to \p$, we see $\bar{f}_*(\n^+) = (\pi_r R^0)_*(\n^+)$.
Since $\mathcal{Q}_r$ is a connected simply-connected nilpotent Lie group, it follows that $\bar{f}(N) = \pi_rR^0(N)$.
\end{proof}
Before beginning the proof of Proposition \ref{nplushom}, we will prove the following:
\begin{prop}\label{nminushom}
If $\iota \in \mathrm{Hom_{gr}}(\mathfrak{n}^{-}, \p)$ is sufficiently close to $\iota^0 = \iota^0 |_{\mathfrak{n}^{-}}\in \mathrm{Hom_{gr}}(\mathfrak{n}^{-}, \p)$, then there is an element $h \in \P_0$ close to the identity such that $\iota^0 = h_* \circ \iota$.
\end{prop}
\begin{proof}
If $\iota$ is sufficiently close to $\iota^0 \in \mathrm{Hom_{gr}}(\mathfrak{n}^{-}, \p)$, we may assume that $\iota\mathfrak{n}^- \subset \mathfrak{X}(\R^{2n+1})$ is a frame field around $\mathrm{O} \in \R^{2n+1}$.
Then $\iota\n^-$ is a Heisenberg connection around $\mathrm{O} \in \R^{2n+1}$.
Thus there is a unique diffeomorphism $f:\R^{2n+1} \to \R^{2n+1}$ such that $f(\mathrm{O}) = \mathrm{O}$ and that $f_*(\iota X)= \iota^0X$ for all $X \in \mathfrak{n}^-$.
We will show that such a diffeomorphism $f$ must be an element of $\P_0 \subset \mathrm{Diff}(\R^{2n+1})$.

Since $\iota$ preserves the gradations, $\iota E$ is a dilation of $\iota \n^-$ by $e$.
As $\iota^0E$ is a dilation of $\iota^0 \n^-$ by $e$, $f_*(\iota^0E)$ is also a dilation of $\iota \n^-$ by $e$.
Thus $\iota E - f_*(\iota^0E) \in Z(\iota \n^-)$.
Using $f(\mathrm{O}) = \mathrm{O}$, we see that $\iota E = f_*(\iota^0E)$.

It remains to show that a diffeomorphism $f$ defined around $\mathrm{O} \in \R^{2n+1}$ fixing $\mathrm{O}$ with $f_*\iota^0E = \iota^0E$ must be in $\P_0$.
Set $f(x) = (f_1(x), \dots, f_{2n+1}(x))\in \R^{2n+1}$.
By $f_*\iota^0E =\iota^0E$,
\[
(\iota^0E)f_i = -f_i\,(1 \leq i \leq 2n), \;(\iota^0E)f_{2n+1} = -2f_{2n+1}.
\]
Thus it is sufficient to show that a smooth function $g$ defined around $\mathrm{O}\in \R^{2n+1}$ with $(\iota^0E)g = mg$, $m \in \mathbb{Z}$ is polynomial.
Recall that $\iota^0E = -x_1\partial_1 - \dots -x_{2n}\partial_{2n} - 2x_{2n+1}\partial_{2n+1} \in \p^{(0)}$.
Observe that $\lim_{t \to \infty}\gamma(t) = \mathrm{O} \in \R^{2n+1}$ for any integral curves $\gamma(t)$ of $\iota^0E$.
It follows that $g= 0$ around $\mathrm{O} \in \R^{2n+1}$ if $m > 0$ and that $g$ is constant around $\mathrm{O} \in \R^{2n+1}$ if $m = 0$.
If $m < 0$, since $[\iota^0E, \partial_i] = \partial_i$, $[\iota^0E, \partial_{2n+1}] = 2\partial_{2n+1}$,
\[
(\iota^0E)\partial_ig = (m+1)\partial_ig, \; (\iota^0E)\partial_{2n+1}g = (m+2)\partial_{2n+1}g
\]
for $1 \leq i \leq 2n$.
Thus we see that $g$ is polynomial function by induction on $m$.
\end{proof}

Let us prove Proposition \ref{nplushom}.
To reduce the notation, put
\[
E = \left(
\begin{array}{ccc}
	1 & 0&0\\
	0& 0 &0\\
	0&0&-1
\end{array} \right),
F^- = \left(
\begin{array}{ccc}
	0 & 0&0\\
	0& 0 &0\\
	-i&0&0
\end{array} \right),
F^+ = \left(
\begin{array}{ccc}
	0 & 0&-i\\
	0& 0 &0\\
	0&0&0
\end{array} \right),
\]
\[
\xi^- = \left(
\begin{array}{ccc}
	0 & 0& 0\\
	\xi& 0 &0\\
	0&-\overline{\xi}^{T}&0
\end{array} \right), \xi^+ =  \left(
\begin{array}{ccc}
	0 & -\overline{\xi}^{T}& 0\\
	0& 0 &\xi\\
	0&0&0
\end{array} \right) \;(\xi \in \mathbb{C}^n).
\]
Then the following relations are immediate:
\[
[\xi^+, \eta^+] = 2\mathrm{Im}(\overline{\xi}^T\eta)F^+,\; [\xi^-, \eta^-] = 2\mathrm{Im}(\overline{\xi}^T\eta)F^-,
\]
\[
[F^-, \xi^+] = (i\xi)^-,\;[F^+, \xi^-] = (i\xi)^+, \;[F^-, F^+] = E.
\]
Observe that $\xi^- = [F^-, (-i\xi)^+]$ and $2F^- = \mathrm{ad}(F^-)^2  F^+$.
Then we can construct a homomorphism of graded Lie subalgebras from $\n^-$ into $\p$ associated with $\iota \in \mathrm{Hom}_\mathrm{gr}(\mathfrak{n}^{+}, \p)$:
\begin{lem}\label{Theta}
For $\iota \in \mathrm{Hom}_\mathrm{gr}(\mathfrak{n}^{+}, \p)$, the linear map $\Theta \iota:\mathfrak{n}^- = \mathfrak{g}^{(-2)}\oplus\mathfrak{g}^{(-1)} \to \p$ defined by
\[
(\Theta \iota)(\xi^-) = \mathrm{ad}(\iota^0F^-)\iota(-i\xi)^+,\; (\Theta \iota)(F^-) = \frac{1}{2}\mathrm{ad}(\iota^0F^-)^2\iota F^+
\]
is a homomorphism of graded Lie algebras.
Moreover, if $\iota \in \mathrm{Hom}_\mathrm{gr}(\mathfrak{n}^{+}, \p)$ is sufficiently close to $\iota^0|_{\n^+} \in \mathrm{Hom}_\mathrm{gr}(\mathfrak{n}^{+}, \p)$,  then  $\Theta\iota \in \mathrm{Hom}_\mathrm{gr}(\mathfrak{n}^{-}, \p)$ is close to $\iota^0|_{\n^-} \in \mathrm{Hom}_\mathrm{gr}(\mathfrak{n}^{-}, \p)$.
\end{lem}
\begin{proof}
By the relations  $\xi^- = [F^-, (-i\xi)^+]$ and $2F^- = \mathrm{ad}(F^-)^2  F^+$, we see that $\Theta \iota$ is close to $\iota^0$ as a linear map from $\mathfrak{n}^-$ to $\p$.
Since $(\Theta \iota)(\mathfrak{g}^{(+1)}) \subset [\p^{(-2)}, \p^{(+1)}] = \p^{(-1)}$ and $(\Theta \iota)(\mathfrak{g}^{(+2)}) \subset [\p^{(-2)},[\p^{(-2)}, \p^{(+1)}]] = \p^{(-2)}$, we see that $\Theta \iota$ preserves the gradations.
Applying $\mathrm{ad}(\iota^0F^-)^2$ to the equation
\[
[\iota(-i\xi)^+, \iota(-i\eta)^+] = 2\mathrm{Im}(\overline{\xi}^T\eta)\iota F^+,
\]
we obtain
\[
2[\mathrm{ad}(\iota^0F^-)\iota(-i\xi)^+, \mathrm{ad}(\iota^0F^-)\iota(-i\eta)^+] = 2\mathrm{Im}(\overline{\xi}^T\eta)\mathrm{ad}(\iota^0F^-)^2\iota F^+.
\]
It follows that $\Theta \iota$ is a homomorphism of Lie algebras.
\end{proof}

\begin{prop}
If a homomorphism $\iota \in \mathrm{Hom}_\mathrm{gr}(\mathfrak{n}^{+}, \p)$ is sufficiently close to $\iota^0|_{\n^+} \in \mathrm{Hom}_\mathrm{gr}(\mathfrak{n}^{+}, \p)$, then  there exist an element $h \in \P_0$ close to the identity such that $\iota^0|_{\n^-} = \Theta(h_*\circ  \iota) \in \mathrm{Hom}_\mathrm{gr}(\mathfrak{n}^{-}, \p)$.
\end{prop}
\begin{proof}
By Lemma \ref{Theta} and Proposition \ref{nminushom}, there is an element $h \in \P_0$ such that $\iota^0 = h_* (\Theta \iota)$.
Then
\begin{align*}
\iota^0(\xi^-) &= \mathrm{ad}(h_*\iota^0 F^-) (h_* \circ \iota)(-i\xi^+), \\
\iota^0(F^-) &= \mathrm{ad}(h_*\iota^0 F^-)^2 (h_* \circ \iota)(F^+).
\end{align*}
There is $t \in \R$ satisfying $h_*\iota^0 F^- = e^t \iota^0F^-$.
Consider an element $h' \in \exp(\R E) \subset \P_0$ such that $h'_*(X) = e^{\lambda t}X$ for all $X \in \mathfrak{g}^{(\lambda)}$.
Then we see that $\Theta((h'h)_*\circ\iota) = \iota^0$.
\end{proof}

Thus replacing $\iota$ by $h_* \circ \iota$ obtained by the above proposition, we may assume that $\Theta \iota = \iota^0|_{\n^-} \in  \mathrm{Hom_{gr}}(\mathfrak{n}^{-}, \p)$.

\begin{lem}
If $\iota \in \mathrm{Hom}_\mathrm{gr}(\mathfrak{n}^{+}, \p)$ is sufficiently close to the homomorphism $\iota^0|_{\n^+} \in \mathrm{Hom}_\mathrm{gr}(\mathfrak{n}^{+}, \p)$, and $\Theta \iota = \iota^0|_{\n^-} \in  \mathrm{Hom}_\mathrm{gr}(\mathfrak{n}^{-}, \p)$, then 
\begin{itemize}
\item $[\iota^0F^-, \iota F^+] = E$,
\item $\iota\mathfrak{g}^{(+1)} = [\iota^0\mathfrak{g}^{(-1)}, \iota \mathfrak{g}^{(+2)}]$.
\end{itemize}
\end{lem}
\begin{proof}
By $[\iota\xi^+, \iota F^+] = 0$, applying $ \mathrm{ad}(\iota^0F^-)^2$,
\[
2[\mathrm{ad}(\iota^0F^-)\iota\xi^+, \mathrm{ad}(\iota^0F^-)\iota F^+] + [\iota\xi^+, \mathrm{ad}(\iota^0F^-)^2\iota F^+]= 0.
\]
Then by $\Theta \iota = \iota^0$, 
\[
[\iota^0(i\xi)^-, \mathrm{ad}(\iota^0F^-)\iota F^+] - \iota^0(i\xi)^-= 0.
\]
It follows that $\mathrm{ad}(\iota^0F^-)\iota F^+ - E \in Z(\iota^0\mathfrak{n}^-)$.
Since $\mathrm{ad}(\iota^0F^-)\iota F^+ - E \in \p^{(0)}$, we see that $\mathrm{ad}(\iota^0F^-)\iota F^+ - E = 0$.
Thus we have proved the first item.

Applying $ \mathrm{ad}(\iota^0F^-)$ to $[\iota\xi^+, \iota F^+] = 0$, 
\[
[\mathrm{ad}(\iota^0F^-)\iota\xi^+, \iota F^+] + [\iota\xi^+, \mathrm{ad}(\iota^0F^-)\iota F^+]= 0.
\]
By $\Theta \iota = \iota^0$ and $[\iota^0F^-, \iota F^+] = E$,
\[
[\iota^0(i\xi)^-, \iota F^+] + [\iota\xi^+, E]= 0.
\]
Thus $\iota\xi^+ = [\iota^0(i\xi)^-, \iota F^+]$ for all $\xi \in \mathbb{C}^n$ and the second item follows.
\end{proof}

By the above lemma, we are left to prove the following:
\begin{prop}
If $X \in \p^{(+2)}$ is sufficiently close to $\iota^0F^+$ satisfying 
\[
[\iota^0 F^-, X] = E, \; [[\iota^0\mathfrak{g}^{(-1)}, X], X] = 0,
\]
then $X = \iota^0F^+$.
\end{prop}
\begin{proof}
To show the proposition, we will use implicit function theorem.
Identifying the tangent space $T_{\iota^0F^+} \p^{(+2)}$ with $\p^{(+2)}$ itself, it is sufficiet to show that $X \in \p^{(+2)}$ satisfying
\begin{align}
[\iota^0 F^-, X] &= 0, \label{norm}\\ 
[[\iota^0\xi^-, X], \iota^0F^+] + [[\iota^0\xi^-,  \iota^0F^+], X] &= 0 \label{rel}
\end{align}
for all $\xi^- \in \mathfrak{g}^{(-1)}$ is $X = 0$.

Define the $\R$-multilinaer map $\Phi: (\mathbb{C}^n)^4 \to \p^{(-2)}$ by
\[
\Phi(\xi_1, \xi_2, \xi_3, \xi_4) = \mathrm{ad}(\iota^0\xi_1^-)\mathrm{ad}(\iota^0\xi_2^-)\mathrm{ad}(\iota^0\xi_3^-)\mathrm{ad}(\iota^0\xi_4^-)X.
\]
By equation (\ref{norm}), we see that $\Phi$ is symmetric with respect to $\xi_i$.
We will show that $\Phi = 0$.
Put $\alpha = \mathrm{ad}(\iota^0\xi^-)$ and $\beta = \mathrm{ad}(\iota^0\eta^-)$ for $\xi, \eta \in \mathbb{C}^n$.
Applying $\beta^5$ to (\ref{rel}),
\[
10[\beta^2\alpha X, \beta^3\iota^0F^+] + 10[\beta^2\alpha \iota^0F^+, \beta^3X]  + 5[\beta\alpha \iota^0F^+, \beta^4X]= 0,
\]
where we used (\ref{norm}) and $\mathrm{ad}(F^-)\mathrm{ad}(\eta^-)^2F^+ = 0$.
When $\eta = -i \xi$, since $\mathrm{ad}(\eta^-)^3F^+ = \|\xi\|^2\xi^-$ and $\mathrm{ad}(\eta^-)\mathrm{ad}(\xi^-)F^+ = \|\xi\|^2E$, it follows that
\[
-10\alpha\beta^2\alpha X +10\beta^4 X + 5[E, \beta^4 X] = 0.
\]
Thus
\begin{equation}\label{PHI}
\Phi(\xi, \xi, i\xi, i\xi) = 0
\end{equation}
for all $\xi \in \mathbb{C}^n$.
Applying $\mathrm{ad}(\iota^0\eta^-)\mathrm{ad}(\iota^0\xi^-)^4$ to (\ref{rel}),
\begin{multline*}
4[\alpha^4 X, \beta\alpha\iota^0F^+] + 6[\alpha^3 X, \beta\alpha^2\iota^0F^+] + 4[\beta\alpha^2 X, \alpha^3\iota^0F^+] \\
+ 6[\alpha^3 \iota^0F^+, \beta\alpha^2X] + 4[\beta\alpha^2 \iota^0F^+, \alpha^3X]  = 0,
\end{multline*}
where we denoted  $\alpha = \mathrm{ad}(\iota^0\xi^-)$, $\beta = \mathrm{ad}(\iota^0\eta^-)$ and used the equations (\ref{norm}) and $\mathrm{ad}(F^-)\mathrm{ad}(\xi^-)^2F^+ = 0$.
When $\eta = -i \xi$, using  $\mathrm{ad}(\eta^-)\mathrm{ad}(\xi^-)F^+ = \|\xi\|^2E$, $\mathrm{ad}(\eta^-)\mathrm{ad}(\xi^-)^2F^+ =\|\xi\|^2 \xi^-$, $\mathrm{ad}(\xi^-)^3F^+ = -\|\xi\|^2\eta^-$, it follows that
\[
4[\alpha^4 X, E] - 6\alpha^4 X+ 4\beta^2\alpha^2 X - 6 \beta^2\alpha^2X + 4\alpha^4X  = 0.
\]
By (\ref{PHI}),
\[
\Phi(\xi, \xi, \xi, \xi) = 0
\]
for $\xi \in \mathbb{C}^n$.
Since $\Phi$ is an $\R$-multilinear symmetric map, it follows that $\Phi = 0$.

Define the $\R$-multilinear map $\Psi: (\mathbb{C}^n)^3 \to \p^{(-1)}$ by
\[
\Psi(\xi_1, \xi_2, \xi_3) = \mathrm{ad}(\iota^0\xi_1^-)\mathrm{ad}(\iota^0\xi_2^-)\mathrm{ad}(\iota^0\xi_3^-)X.
\]
By equation (\ref{norm}), $\Psi$ is symmetric .
We will show that $\Psi = 0$.
Applying $\mathrm{ad}(\iota^0\xi^-)^4$ to (\ref{rel}), since $\Phi =0$,
\[
[\Psi(\xi, \xi, \xi), \iota^0J] + \Psi(i\xi, \xi, \xi) = 0,
\]
where
\[
J = \left(\begin{array}{ccc}
-i& 0& 0\\
0&0&0\\
0&0&i
\end{array}\right) \in \mathfrak{g}^{(0)}.
\]
As $\Psi$ is an $\R$-multilinear symmetric map,
\begin{equation}\label{PsiJ}
3[\Psi(\xi, \eta, \zeta), \iota^0J] + \Psi(i\xi, \eta, \zeta)+ \Psi(\xi, i\eta, \zeta) + \Psi(\xi, \eta, i\zeta)  = 0
\end{equation}
for $\xi, \eta, \zeta \in \mathbb{C}^n$.
Applying $\mathrm{ad}(\iota^0\eta^-)^4$ to (\ref{rel}), when $\eta = -i\xi$,
\[
-4\Psi(\xi, \eta, \xi) + 6[\Psi(\eta, \eta, \xi), \iota^0J] +2\Psi(\eta, \eta, \eta) = 0.
\]
By (\ref{PsiJ}),
\[
\Psi(\eta, \eta, \eta) = 0.
\]
Thus $\Psi = 0$.

As we have seen in Section \ref{Heis}, the centralizer $Z(\iota^0\n^-)$ of $\iota^0\n^-$ in $\mathfrak{X}(\R^{2n+1})$ is contained in $\p^{(-2)}\oplus\p^{(-1)}$.
Since $\Psi$ is identically $0$, for all $\xi_1, \xi_2 \in \mathbb{C}^n$, $\mathrm{ad}(\iota^0\xi_1^-)\mathrm{ad}(\iota^0\xi_2^-)X \in Z(\iota^0\n^-) \subset \p^{(-2)}\oplus\p^{(-1)}$.
On the other hand, since $\iota^0\xi_1^-, \iota^0\xi_2^- \in \p^{(-1)}$ and $X \in \p^{(+2)}$, it follows that $\mathrm{ad}(\iota^0\xi_1^-)\mathrm{ad}(\iota^0\xi_2^-)X \in \p^{(0)}$.
So $\mathrm{ad}(\iota^0\xi_1^-)\mathrm{ad}(\iota^0\xi_2^-)X = 0$.
Similarly, we see that $\mathrm{ad}(\iota^0\xi^-)X = 0$ for all $ \xi \in \mathbb{C}^n$, and that $X = 0$.
\end{proof}

\end{document}